\documentclass[12pt]{amsart}

\pdfoutput=1
\usepackage{latexsym}
\usepackage[centertags]{amsmath}
\usepackage{amsfonts}
\usepackage{amssymb}
\usepackage{amsthm}
\usepackage{newlfont}
\usepackage{graphics}
\usepackage{color}
\usepackage{booktabs}
\usepackage{wrapfig}
\usepackage{subfigure}
\usepackage{graphicx}
\usepackage[usenames,dvipsnames]{xcolor}
\usepackage{graphicx}

\usepackage{wrapfig}
\usepackage[nocompress]{cite}
\usepackage[latin1]{inputenc}
\makeatletter
\newcommand{\tabcaption}{\def\@captype{table}\caption}
\makeatother

\newtheorem{theo}{Theorem}

\newtheorem{lem} [theo]{Lemma}
\newtheorem{cor}[theo]{Corollary}
\newtheorem{prop}[theo]{Proposition}

\makeatletter \@addtoreset{equation}{section}
\@addtoreset{theo}{}\makeatother

\numberwithin{equation}{section}

\def\Oge{\mathop{\Omega}_{\geq}}

\def\N{\mathbb{N}}
\def\Z{\mathbb{Z}}

\def\x{\mathbf{x}}
\def\P{\mathbb{P}}
\def\a{{\pmb{\alpha}}}
\def\b{{\pmb{\beta}}}
\def\c{{\pmb{\gamma}}}

\newcommand{\rank}{\operatorname{rank}}

\def\N{\mathbb{N}}
\def\Z{\mathbb{Z}}
\def\R{\mathbb{R}}

\title{On Magic Distinct Labellings of Simple Graphs}

\author{Guoce Xin$^{1,*}$, Xinyu Xu$^{2}$, Chen Zhang$^{3}$ and Yueming Zhong$^{4}$}

 \address{ $^{1,2,3,4}$School of Mathematical Sciences, Capital Normal University,
 Beijing 100048, PR China}
\email{$^1$\texttt{guoce\_xin@163.com}\ \& $^2$\texttt{xinyu0510x@163.com} \ \& $^3$\texttt{chenzh7663@gmail.com}\ \& $^4$\texttt{zhongyueming107@gmail.com}}

\date{ \today}
\thanks{$*$ This work was partially supported by NSFC(12071311).}

\begin{document}

\begin{abstract}
A magic labelling of a graph $G$ with magic sum $s$
is a labelling of the edges of $G$ by nonnegative integers such that for each vertex $v\in V$,
the sum of labels of all edges incident to $v$ is equal to the same number $s$. Stanley
gave remarkable results on magic labellings, but the distinct labelling case is much more complicated.
We consider the complete construction of all magic labellings of a given graph $G$. The idea
is illustrated in detail by dealing with three regular graphs. We give combinatorial proofs.
The structure result was used to enumerate the corresponding magic distinct labellings.
\end{abstract}

\maketitle

\vspace{-5mm}
\maketitle

\noindent
\begin{small}
 \emph{Mathematic subject classification}: Primary 05A19; Secondary 11D04; 05C78.
\end{small}

\noindent
\begin{small}
\end{small}

\section{Introduction}
Throughout this paper, we use standard set notations $\R, \Z, \N, \P$ for real numbers, integers, nonnegative integers, and positive integers respectively.

Let $G=(V,E)$ be a finite (undirected) graph with vertex set $V=\{v_1,\dots,v_m\}$ and edge set $E=\{a_1,\dots, a_n\}$.
A magic labelling of $G$
is a labelling of the edges in $E$ by nonnegative integers such that for each vertex $v\in V$,
the weight $wt(v)$ of $v$, defined to be the sum of labels of all edges incident to $v$, is equal to the same number $s$, called magic sum (also called
index by some authors).
More precisely, let $\mu: E\mapsto \N$ be the labelling. Then
\begin{align}
  \label{e-magic-sum}
  wt(v_i):= \sum_{j=1,(v_i,v_j)\in E }^m   \mu(v_i,v_j) =s, \qquad i=1,2,\dots, m.
\end{align}
A magic distinct labelling is a magic labelling whose labels are distinct. It is said to be pure if the labels are
$1,2,\dots, n$.
These concepts were introduced by graph theorists
as an analogous of magic squares,
which have been objects of study for centuries.

Plenty work has been done for graph labellings. Magic labellings of simple graphs seem first introduced in  \cite{MacDougall} as vertex magic
labellings. Vertex magic total labelling of a simple graph indeed corresponds to our magic labelling of the same graph with a loop attached to each vertex. For related research, see \cite{Nissankara,Matthias1,Matthias2,Arnold,Kotzig,Baker2,xin2}. Note that
``magic" may have different meaning in different context.

In the 1970s, Stanley \cite{Stanley-magiclabelings} proved some remarkable facts for
magic labellings:
\begin{theo}
Let $G$ be a finite graph and define $h_G(s)$ to be the number of magic labellings of
$G$ of index $s$. There exist polynomials $P_G(s)$ and $Q_G(s)$ such that $h_G(s) = P_G(s) + (-1)^s Q_G(s)$. Moreover,
if the graph obtained by removing all loops from $G$ is bipartite, then $Q_G(s) = 0$, i.e., $h_G(s)$ is a polynomial
of $s$.
\end{theo}
In terms of generating functions, the theorem asserts that $F^G(y)=\sum_{s\ge 0} H_G(s) y^s$ is a rational function
with denominator factors $1\pm y$.

Though magic labellings of graphs behave nicely, magic distinct labellings of graphs behave very badly
because of the ``distinct" condition on the labels. For instance, for the graph $G_4$ depicted in Figure \ref{fig:RegularGraph6-9-2},
the generating function for magic labellings is $\frac{1+y+{y}^{2}} {\left({1-y}\right)^{5}}$, but
the generating function for magic distinct labellings is
$$\frac{72  y ^{12}\left({1-y}\right)^{2} N_4(y)} {\left({1-{y}^{3}}\right)^{2}\left({1-{y}^{4}}\right)\left({1-{y}^{5}}\right)\left({1-{y}^{6}}\right)\left({1-{y}^{7}}\right)\left({1-{y}^{8}}\right)},$$
where $N_4(y)$ is a polynomial of degree 19. See \eqref{G4-Fy} and \eqref{G4Egenerationfuction} for details.

Our starting point is a simple structure
result for magic labellings, from which we are able to extract information about magic distinct labellings.

The set of magic $\R$-labellings of $G$ is
$$S_\R(G):=\{\a=(\alpha_1,\dots, \alpha_n)\in \R^{n}: \eqref{e-magic-sum} \text{ holds for } \mu(a_k)=\alpha_k, \ 1\le k \le n, \ s=wt(v_1) \}.$$
Clearly, $S_\R(G)$ is a subspace of $\R^{n}$
and its basis can be easily computed by linear algebra. Even the structure of $S_\Z(G)=S_\R(G)\cap \Z^{n}$ is easy: it is a finitely generated abelian group, and there are known algorithms for finding the generators. But the set of magic labellings $S(G)=S_{\R}(G)\cap \N^{n}$
only forms a (commutative) monoid (semi-group with identity), and it is usually not the case
that the monoid is free; that is, there exists $\a_1,\dots, \a_d \in S(G)$, called generators, such that every $\a\in S(G)$ can be written
uniquely as $\sum_{i=1}^d k_i \a_i$ where $k_i\in \N$.

We will decompose $S(G)$ into some shifted free monoids, whose elements can be uniquely written as
$\gamma + \sum_{i=1}^d k_i \a_i$, where $\gamma$ is fixed and usually in $S(G)$,
and $\a_1,\dots, \a_d$ are still called the generators. We illustrate the idea by three examples.
We give combinatorial proofs.

In terms of generating functions, we define
$$F^G(\x)=\sum_{\a\in S_\N(G)} \x^\a,$$
where $\x^\a$ is short for $x_1^{\alpha_1}\cdots x_n^{\alpha_n}$.
It is known that $F^G(\x)$ is a rational function with denominator $\prod_{\beta} (1-x^\beta)$
where $\beta$ ranges over all extreme rays of $S(G)$. See \cite[Theorem 4.6.11]{Stanley-E-Combinatiorics1}.
There are existing algorithms for computing $F^G(\x)$, such as
the Mathematica package \texttt{Omega} in \cite{Andrews-Omega}, the Maple packages \texttt{Ell} in \cite{xin-fast}
and \texttt{CTEuclid} in \cite{xin-Euclid}. But the representation of $F^G(\x)$ by computer is usually not ideal.

Our decomposition give rise a rational function decomposition:
$F^G(\x)=F_1(\x)+F_2(x)+\cdots$, where each
$F_j(\x)$ (corresponding to a shifted free monoid) is of the form
$$\frac{\x^\gamma}{(1-\x^{\a_1})(1-\x^{\a_2})\cdots (1-\x^{\a_d})}.$$

The paper is organized as follows. Section 1 is this introduction.
In Section \ref{sec:MainResult} we deal with three graphs $G_1$, $G_2$, and $G_3$. We give detailed construction of their magic labellings,
and combinatorial proofs.
In Section \ref{sec:BasicIdeas} we introduce basic idea of MacMahon's partition analysis, outline the result for $G_4$,
and setup basic tools for attacking magic distinct labellings of graphs. In particular, we compute
the generating functions for several graphs.

\section{Three examples}\label{sec:MainResult}
In this section, we illustrate our decomposition by considering the magic labellings of three graphs
depicted in Figures \ref{fig:RegularGraph4-6},\ref{fig:RegularGraph6-9} and \ref{fig:RegularGraph8-12}.
These graphs are all regular, i.e., the degree for each vertex is the same.

In what follows we will often use $e_i$ to denote the $i$-th unit vector in $\R^n$. Then $\a=(\alpha_1,\dots, \alpha_n)\in \R^n$
will be written as $\a=\sum_{i=1}^{n} \alpha_i e_i$. The dimension $n$ will be clear from the context.

For $\a\in S(G)$, the magic sum $s=s(\a)$ is determined by $\a$, so it can be treated as a
redundant variable. It is convenient for us to use the generating function
$$F^G(\x,y)=\sum_{\a\in S_\N(G)} \x^\a y^{s(\a)}.$$
Then $F^G(\x)=F^G(\x,1)$ and setting $x_i=1$ for all $i$ gives the enumerating generating function
$$ F^G(y)=F^G(\mathbf{1},y)= \sum_{s\ge 0} h_G(s) y^s,$$
where $h_G(s)$ counts the number of magic labellings of $G$ with magic sum $s$.

\subsection{Example 1}
Let $G_1=K_4$ be the complete graph with
$V=\{v_i,i=1,2,\cdots,4\}$ and $E=\{a_i,i=1,2,\cdots,6\}$ as shown in Figure \ref{fig:RegularGraph4-6}.
\begin{figure}[!ht]
\centering{
\includegraphics[height=1.3 in]{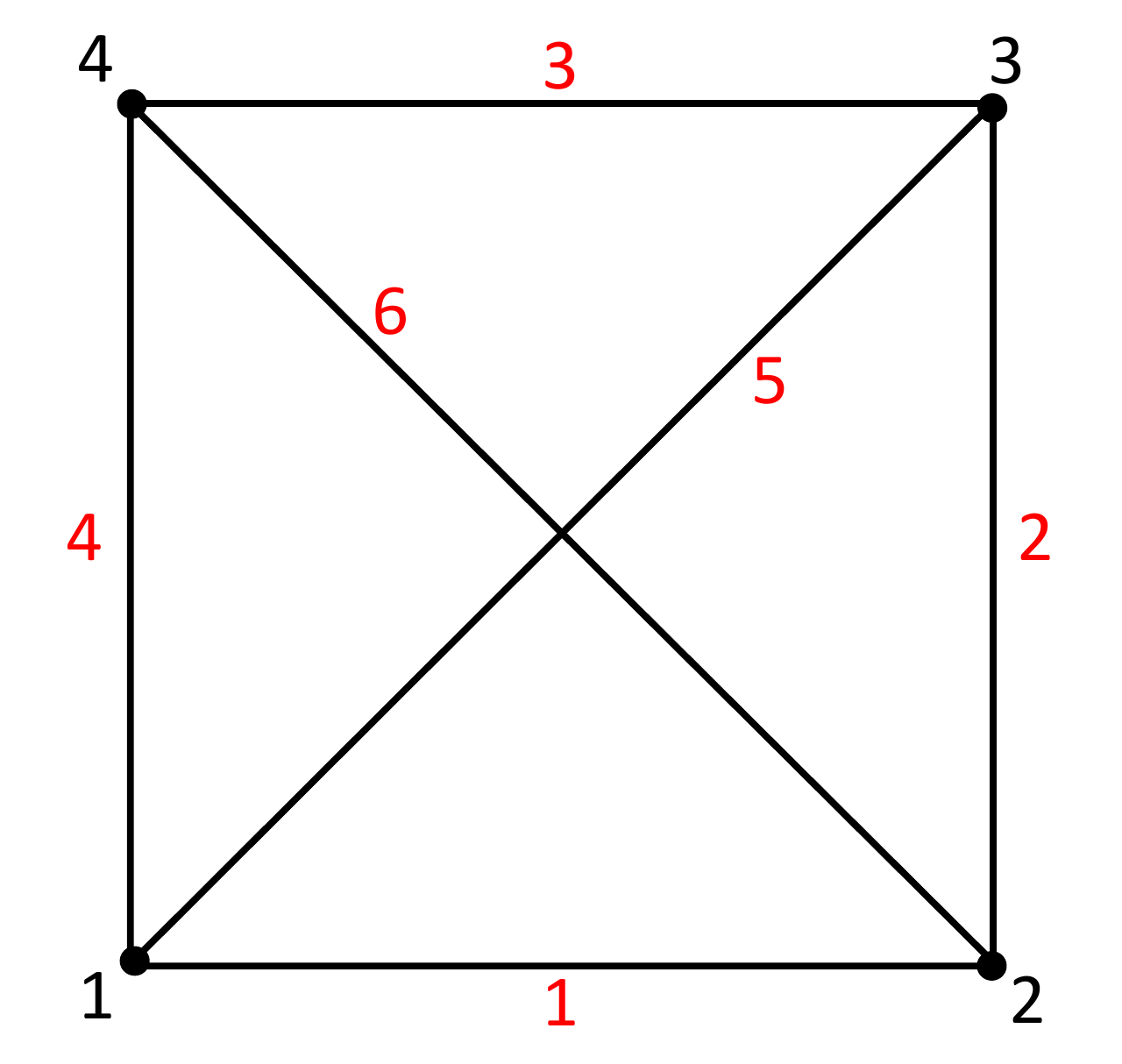}}\vspace{-0.5cm}
\caption{The complete graph $G_1=K_4$
}\label{fig:RegularGraph4-6}
\end{figure}

Suppose the magic labelling is given by $\mu(a_i)=x_i,\; i=1,2,\dots, 6$. Then the fulfilled equations \eqref{e-magic-sum} becomes
\begin{equation}\label{eqRegularGraph4-6}
               \left\{
                \begin{aligned}
                      x_1+x_4+x_5&=s,\\
                      x_1+x_2+x_6&=s,\\
                      x_2+x_3+x_5&=s,\\
                      x_3+x_4+x_6&=s.\\
                \end{aligned}\right.
\end{equation}
Then $$ \a \in S_\R(G_1) \Leftrightarrow A (\a,s)^T=0, \ A= \begin{pmatrix}
                                                                     1  &0   &0   &1   &1    &0 &-1\\
                                                                     1  &1   &0   &0   &0    &1&-1\\
                                                                     0  &1   &1   &0   &1    &0&-1\\
                                                                     0  &0   &1   &1   &0    &1&-1
                                                                     \end{pmatrix}.$$
Thus $S_\R(G_1)$ can be identified with the null space of $A$, which is a subspace in $\R^{7}$.
Since $\rank(A)=4$, $\dim S_\R(G_1)=7-4=3$.

Indeed, we have the following result.
\begin{prop}
  For $G_1$ as above, the magic labelling of $G_1$ forms a free monoid with generators
  $\a_1=e_2+e_4,\a_2=e_1+e_3,\a_3=e_5+e_6$.
\end{prop}
\begin{proof}
It is straightforward to check that they are linearly independent and hence form a basis of $S_\R(G_1)$.

Let $\a=k_1\a_1+k_2\a_2+k_3\a_3=k_2e_1+k_1e_2+k_2e_3+k_1e_4+k_3e_5+k_3e_6\in S_\R(G_1) $. Then
$\a\in S(G_1)$ if and only if $k_1,k_2,k_3\in\mathbb{N}$. Therefore
$\a_1, \a_2$ and $\a_3$ freely generate $S(G_1)$.
\end{proof}
Note that $\a_1,\a_2,\a_3$ correspond to perfect matchings of $G_1$.

\begin{cor}
Let $G_1$ be as above. Then
$$F^{G_1}(\x,y)=\frac{1} {\left({1-yx_{{2}}x_{{4}}}\right)\left({1-yx_{{1}}x_{{3}}}\right)\left({1-yx_{{5}}x_{{6}}}\right)}.$$
Consequently,
\begin{align*}
  F^{G_1}(y) &=\frac{1} {\left({1-y}\right)^{3}}=1+3\,y+6\,{y}^{2}+10\,{y}^{3}+15\,{y}^{4}+21\,{y}^{5}+28\,{y}^{6}+36\,
{y}^{7}+\cdots.
\end{align*}
\end{cor}

\subsection{Example 2}
Let $G_2$ be the graph with
$V=\{v_i,i=1,2,\cdots,6\}$ and $E=\{a_i,i=1,2,\cdots,9\}$ as shown in Figure \ref{fig:RegularGraph6-9}.
\begin{figure}[!ht]
  \centering{
  \includegraphics[height=1.3 in]{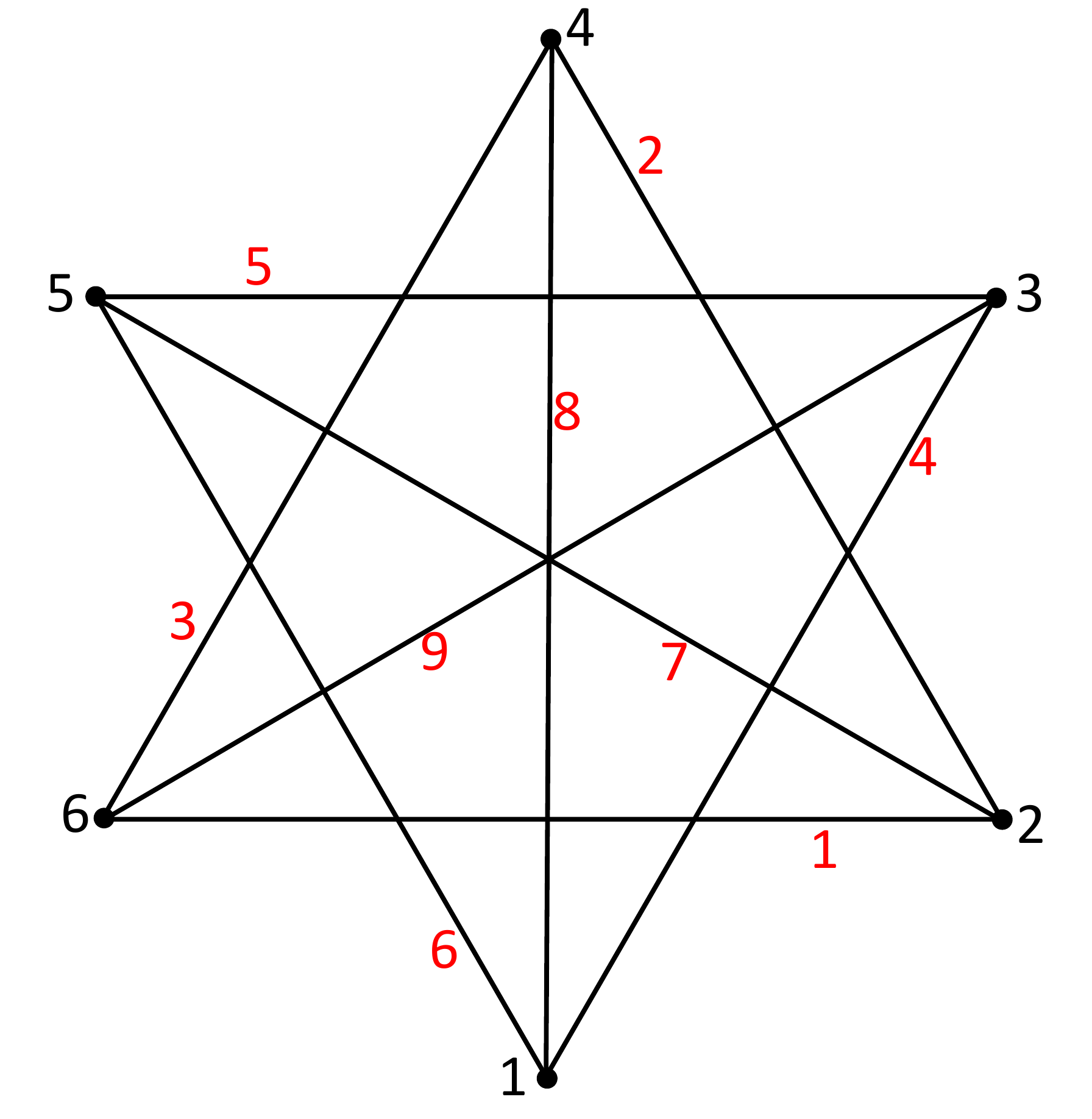}}\vspace{-0.5cm}\\
  \caption{The Graph $G_2$ with 6 vertices and 9 edges.}\label{fig:RegularGraph6-9}
\end{figure}
Suppose the magic labelling is given by $\mu(a_i)=x_i,\; i=1,2,\dots, 9$. Then the fulfilled equations \eqref{e-magic-sum} becomes
\begin{equation}\label{eqRegularGraph6-9}
              \left\{
              \begin{aligned}
                     x_4+x_6+x_8&=s,\\
                     x_1+x_2+x_7&=s,\\
                     x_4+x_5+x_9&=s,\\
                     x_2+x_3+x_8&=s,\\
                     x_5+x_6+x_7&=s,\\
                     x_1+x_3+x_9&=s.\\
              \end{aligned}\right.
\end{equation}
It is easy to see the following properties hold:

i) $\dim S_\R(G_2)=4$;

ii)
$\begin{array}{llll}\label{beq2}
\b_1=e_1+e_5+e_8, &\b_2=e_2+e_6+e_9, &\b_3=e_3+e_4+e_7, &
\b_4=e_7+e_8+e_9
\end{array}$
are linearly independent and hence form a basis of $S_\R(G_2)$, and
they correspond to the 4 perfect matchings of $G_2$;

iii) $\b_5=e_1+e_2+e_3+e_4+e_5+e_6\in S(G_2)$ is not a nonnegative linear combination of $\b_1,\dots, \b_4$.
Indeed we have $\b_5=\b_1+\b_2+\b_3-\b_4$.

\begin{prop}\label{pro:6-9}
        Let $G_2$ be as in Figure 2. Then
        every $\b$ in $S(G_2)$ can be uniquely written in one of the following two types.
        \begin{enumerate}
               \item [$B_1:$]  $l_1\b_1+l_2\b_2+l_3\b_3+l_4\b_4,\  l_i\in \N\ (1 \leqslant i \leqslant 4);$
               \item [$B_2:$]  $\b_5+ l_1\b_1+l_2\b_2+l_3\b_3+l_4\b_5,\  l_i\in \N\ (1 \leqslant i \leqslant 4).$
        \end{enumerate}
\end{prop}
\begin{proof}
By property ii), any element $\b \in S_\R(G_2)$ can be written as
$\b=k_1\b_1+k_2\b_2+k_3\b_3+k_4\b_4$ where $k_i\in \mathbb{R}, \ 1 \leqslant i \leqslant 4$.

 Then $\b=k_1e_1+k_2e_2+k_3e_3+k_3e_4+k_1e_5+k_2e_6+(k_3+k_4)e_7+(k_1+k_4)e_8+(k_2+k_4)e_9$.
 It belongs to $S(G_2)$ if and only if
$k_1,k_2,k_3,k_1+k_4,k_2+k_4,k_3+k_4\in\N $. Note that we can only deduce that $k_4\in \mathbb{Z}$.

When $k_4\in\mathbb{N}$, such $\b$ naturally corresponds to the type $B_1$ case (by setting $l_i=k_i\in \mathbb{N}\ (1 \leqslant i \leqslant 4)$).

When $k_4<0$, i.e., $-k_4 \in \mathbb{P}$ we need to rewrite
  \begin{align*}
         \b&=k_1\b_1+k_2\b_2+k_3\b_3+k_4(\b_1+\b_2+\b_3-\b_5)\\
               &=\b_5+(k_1+k_4)\b_1+(k_2+k_4)\b_2+(k_3+k_4)\b_3+(-k_4-1)\b_5.
\end{align*}
By comparing with the type $B_2$ case, we shall
have $l_i=k_i+k_4$ for $i=1,2,3$ and $l_4=-k_4-1$.
The conditions on the $k_i$'s transform exactly to $l_i\in \mathbb{\N}, \; 1\le i\le 4$.
Finally, the uniqueness follows by the linear independency of $\b_1,\b_2,\b_3,\b_5$.
\end{proof}

\begin{cor}
Let $G_2$ be as above. Then

$\begin{aligned}
            F^{G_2}(\x,y)=&\frac{1}{(1-yx_1x_5x_8)(1-yx_2x_6x_9)(1-yx_3x_4x_7)(1-yx_7x_8x_9)}\\
                         &+\frac{y^2x_1x_2x_3x_4x_5x_6}{(1-yx_1x_5x_8)(1-yx_2x_6x_9)(1-yx_3x_4x_7)(1-y^2x_1x_2x_3x_4x_5x_6)}
\end{aligned}.$ 

Consequently,
\begin{align*}
  F^{G_2}(y) &=\frac{1+y+{y}^{2}} {\left({1-y}\right)^{3}\left({1-{y}^{2}}\right)}=1+4\,y+11\,{y}^{2}+23\,{y}^{3}+42\,{y}^{4}+69\,{y}^{5}+106\,{y}^{6}+\cdots.
\end{align*}
\end{cor}

\subsection{Example 3}
Let $G_3$ be the graph with
$V=\{v_i,i=1,2,\cdots,8\}$ and $E=\{a_i,i=1,2,\cdots,12\}$ as shown in Figure \ref{fig:RegularGraph8-12}.
\begin{figure}[h]
  \centering
  \includegraphics[width=4cm]{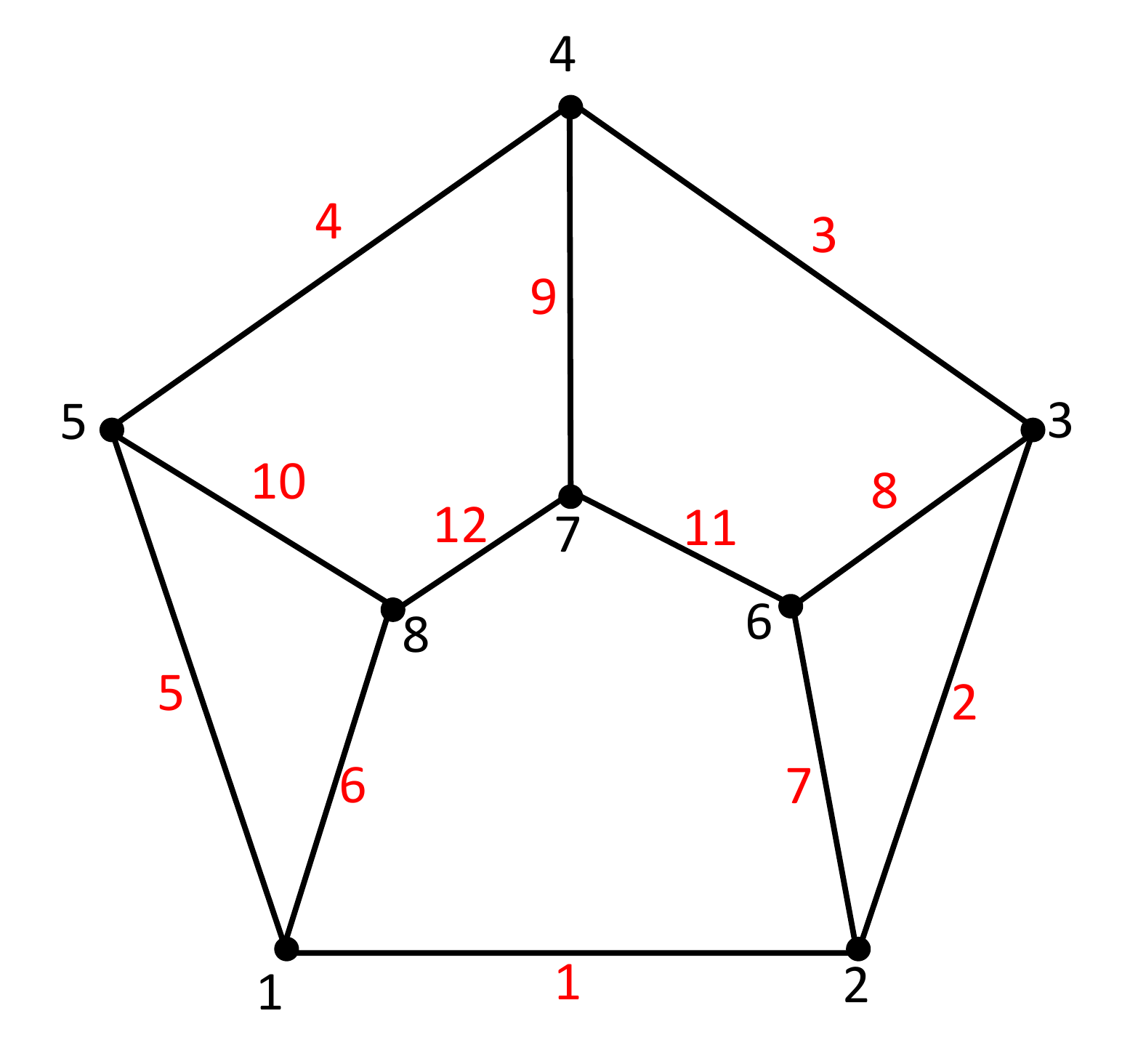}\vspace{-0.5cm}\\
  \caption{The Graph $G_3$ with 8 vertices and 12 edges.}\label{fig:RegularGraph8-12}
\end{figure}
Suppose the magic labelling is given by $\mu(a_i)=x_i,\; i=1,2,\dots, 12$. Then the fulfilled equations \eqref{e-magic-sum} become
\begin{equation}\label{eqRegularGraph8-12}
              \left\{
              \begin{aligned}
              x_4+x_5+x_6&=s,\\
              x_1+x_2+x_7&=s,\\
              x_2+x_3+x_8&=s,\\
              x_3+x_4+x_9&=s,\\
              x_4+x_5+x_{10}&=s,\\
              x_7+x_8+x_{11}&=s,\\
              x_9+x_{11}+x_{12}&=s,\\
              x_6+x_{10}+x_{12}&=s.\\
              \end{aligned}\right.
\end{equation}

It is easy to see that $\dim S_\R(G_3)=5$.
The structure of $S(G_3)$ is indeed pretty complicated. Our combinatorial proof is guided by but independent of the algebraic decomposition described in
Section \ref{sec:BasicIdeas}.

We need the following vectors
\begin{footnotesize}
\begin{align}\label{ceq3}
  \begin{array}{ll}
    \c_1=e_1+e_3+e_{10}+e_{11}, & \c_2=e_2+e_4+e_6+e_{11},\\
    \c_3=e_3+e_5+e_7+e_{12},    & \c_4=e_2+e_5+e_6+e_7+e_8+2e_9+e_{10},\\
    \c_5=e_2+e_3+e_5+e_6+e_7+e_9+e_{10}+e_{11}, &\c_6=e_1+e_4+e_8+e_{12},  \\
    \c_7=e_2+e_4+e_5+e_6+e_7+e_8+e_9+e_{12},&\c_8=e_1+e_8+e_9+e_{10}.
  \end{array}
\end{align}
\end{footnotesize}
The $\c_i$'s are the extreme rays which are computed by other methods.
The following relations can be easily checked.
\begin{prop}
The $\c_i(1\leqslant i \leqslant 8)$ have the following relationship.
\begin{equation}\label{tab:gammarelation}
\begin{tabular}{l|llrr}
\toprule
$(b)$        &$\c_5=\c_1+\c_4-\c_8 $  \\
\midrule
$(c_1)$       &$\c_4=-\c_1+\c_2+\c_3-\c_6+2\c_8$ \\
           &$\c_5=\c_2+\c_3-\c_6+\c_8 $    \\
\midrule
$(c_2)$       &$\c_1=\c_2+\c_3-\c_6+2\c_8 $   \\
           &$\c_5=\c_2+\c_3-\c_6+\c_8 $    \\
\midrule
$(c_3)$       &$\c_1=\c_2+\c_3+\c_4+\c_6-2\c_7$   \\
           &$\c_5=\c_2+\c_3+\c_4-\c_7$        \\
\bottomrule
\end{tabular}
\end{equation}
\end{prop}

In order to the Theorem \ref{theo:8-12}, we give the following Lemma.
\begin{lem}\label{lem:8-12}
Any  $\c\in S_\R(G_3)$ can be uniquely written as $\c=k_1\c_1+k_2\c_2+k_3\c_3+k_4\c_4+k_5\c_5$
for some $k_i\in \mathbb{R}\ (1 \leqslant i \leqslant 5)$. It belongs to $S(G_3)$
  if and only if the following properties hold true.
\begin{equation}\label{tab:kproperty}
\begin{tabular}{llll|rrrr}
\toprule
$k_1 $,                  &$k_2  $,               &$k_3 $,                 &$k_4  $,         &  \\
$k_1+k_2+k_5  $,         &$k_1+k_3+k_5 $,        &$k_1+k_4+k_5 $,         &$k_2+k_4+k_5 $,  & $\in \mathbb{N}$\\
$k_3+k_4+k_5$,           &$2k_4+k_5 $.           &                        &                 &\\
\bottomrule
\end{tabular}
\end{equation}
Consequently, $k_5 \in \mathbb{Z}$.
\end{lem}
\begin{proof}
The first part follows by the linear independency of $\c_1,\dots, \c_5$. By writing in the $e_i$ basis, we have
  \begin{align}\label{e-gamma}
  \c=&k_1e_1+(k_2+k_4+k_5)e_2+(k_1+k_3+k_5)e_3+k_2e_4 \nonumber\\
           &\ +(k_3+k_4+k_5)e_5+(k_2+k_4+k_5)e_6 \\
           &\ \ +(k_3+k_4+k_5)e_7+k_4e_8+(2k_4+k_5)e_9 \nonumber\\
           &\ \ \ +(k_1+k_4+k_5)e_{10}+(k_1+k_2+k_5)e_{11}+k_3e_{12}.\nonumber
  \end{align}
Now $\c\in S(G_3)$ if and only if each coordinate belongs to $\N$. This is the second part.
The consequence $k_5 \in \mathbb{Z}$ is obvious.
\end{proof}

From now on, we will identify $\c\in S_\R(G_3)$ with $(k_1,\dots, k_5)\in \R^5$.
For a set $S$ of $(k_1,\dots, k_5)\in \R^5$ and a property $P$, we will use $P(S)$ to denote the subset of $S$ satisfying $P$.
We will use the following properties.
\begin{description}
  \item[$P$]    \ \ \ \ \ \  $k_i \in \N(1 \leqslant i \leqslant 4), k_5 \in \Z;$
  \item[$P_a$]    \ \ \ \ \  $k_5 \in \N,k_i \in \N (1\leqslant i \leqslant 4);$
  \item[$P_b$]           \ \ $-k_5\in\mathbb{P}, \ \min\{k_1,k_4\}\geqslant -k_5, k_i \in \N (1\leqslant i \leqslant 4);$
  \item[$P_{c_1}$]           $-k_5\in\mathbb{P}, \ \min\{k_1,k_4\} <   -k_5,   \  k_1\geqslant k_4,k_i \in \N (1\leqslant i \leqslant 4);$
  \item[$P_{c_2}$]           $-k_5\in\mathbb{P}, \ \min\{k_1,k_4\} <   -k_5,   \  k_1 <   k_4,    \   2k_1+k_5\geqslant 0,k_i \in \N (1\leqslant i \leqslant 4);$
  \item[$P_{c_3}$]           $-k_5\in\mathbb{P}, \ \min\{k_1,k_4\} <   -k_5,   \  k_1 <   k_4,    \   2k_1+k_5<0,k_i \in \N (1\leqslant i \leqslant 4).$
\end{description}
It is easy to see that
$P(\R^5)=P_a(\R^5)\uplus P_b(\R^5)  \uplus P_{c_1}(\R^5)\uplus P_{c_2}(\R^5) \uplus P_{c_3}(\R^5),$ where
$\uplus$ denotes disjoint union.
This implies that
$$P(S(G_3))=P_a(S(G_3))\uplus P_b(S(G_3))  \uplus P_{c_1}(S(G_3))\uplus P_{c_2}(S(G_3)) \uplus P_{c_3}(S(G_3)). $$

\begin{lem}\label{lem:8-12-2}
Let $h$ be any of $a,b,c_1,c_2$ and $c_3$. Then $P_h(S(G_3))=\{(l_1,\dots, l_5)\in \N^5: l_i \textrm{ are given in Table }\ref{tab:lvalue1}\} $.
\begin{equation}\label{tab:lvalue1}
\begin{tabular}{l|l|l|l|l|l}
\toprule
Types                     &$l_1$            &$l_2$                &$l_3$             &$l_4$           &$l_5$         \\
\hline
$P_a(S(G_3))$             &$k_1$            &$k_2$                & $k_3$            &$k_4$           &$k_5$       \\
\hline
$P_b(S(G_3))$             &$k_1+k_5$        &$k_2$                &$k_3$             &$k_4+k_5$       &$-k_5-1$  \\
\hline
$P_{c_1}(S(G_3))$         &$k_1-k_4$        &$k_2+k_4+k_5$        &$k_3+k_4+k_5$     &$-k_5-k_4-1$    &$2k_4+k_5$   \\
\hline
$P_{c_2}(S(G_3))$         &$k_2+k_1+k_5$    &$k_3+k_1+k_5$        &$k_4-k_1-1$       &$-k_5-k_1-1$    &$2k_1+k_5$  \\
\hline
$P_{c_3}(S(G_3))$         &$k_2+k_1+k_5$    &$k_3+k_1+k_5$        &$k_4+k_1+k_5$     &$k_1$           &$-k_5-2k_1-1$  \\
\bottomrule
\end{tabular}
\end{equation}
\end{lem}
\begin{proof}
We only prove the case of $h=c_1$, namely, $P_{c_1}(S(G_3))=\{(l_1,\dots, l_5)\in \N^5:l_1=k_1-k_4, l_2=k_2+k_4+k_5, l_3=k_3+k_4+k_5, l_4=-k_5-k_4-1, l_5=2k_4+k_5\} $. Other cases are similar.

$``\subseteqq"$
If $(k_1,k_2,k_3,k_4,k_5)\in P_{c_1}(S(G_3))$, then we have $-k_5\in\mathbb{P}, \ \min\{k_1,k_4\} <   -k_5,   \  k_1\geqslant k_4$ and $k_i \in \N (1\leqslant i \leqslant 4).$ We get $l_1=k_1-k_4 \in \N$ and $l_4=-k_5-k_4-1 \in \N$. In addition, by Table \ref{tab:kproperty} in Lemma \ref{lem:8-12}, we get $l_2=k_2+k_4+k_5, l_3=k_3+k_4+k_5$ and $l_5=2k_4+k_5$ are all in $\N$.

$``\supseteqq"$
If $(l_1,l_2,l_3,l_4,l_5) \in \{(l_1,\dots, l_5)\in \N^5:l_1=k_1-k_4, l_2=k_2+k_4+k_5, l_3=k_3+k_4+k_5, l_4=-k_5-k_4-1, l_5=2k_4+k_5\} $, then we get $l_1=k_1-k_4\in \N$ and $l_4=-k_5-k_4-1\in \N$. So $k_1\geqslant k_4$ and $\min\{k_1,k_4\}=k_4$. Also, by inversely solving the $k_i$'s from the $l_i$'s, we obtain
$k_5=-2l_4-l_5-2<0$. Hence, $-k_5\in\mathbb{P}$ and $\min\{k_1,k_4\}=k_4 < -k_5$. Thus, $(k_1,k_2,k_3,k_4,k_5)\in P_{c_1}(S(G_3))$.
\end{proof}

Now we are ready to state and prove our result.
\begin{theo}\label{theo:8-12}
Let $G_3$ be the graph in Figure 3. Then any $\c\in S(G_3)$ can be uniquely written in one of the following five types.
   \begin{enumerate}
     \item [$T_a:$]  $l_1\c_1+l_2\c_2+l_3\c_3+l_4\c_4+l_5\c_5;$
     \item [$T_b:$]  $\c_8+l_1\c_1+l_2\c_2+l_3\c_3+l_4\c_4+l_5\c_8;$
     \item [$T_{c_1}:$] $\c_6+l_1\c_1+l_2\c_2+l_3\c_3+l_4\c_6+l_5\c_8;$
     \item [$T_{c_2}:$] $\c_4+\c_6+l_1\c_2+l_2\c_3+l_3\c_4+l_4\c_6+l_5\c_8;$
     \item [$T_{c_3}:$] $\c_7+l_1\c_2+l_2\c_3+l_3\c_4+l_4\c_6+l_5\c_7,$
   \end{enumerate}
   where $\c_j,\ 1\le j \le 8$ are given by Equation \eqref{ceq3} and $l_i \in \mathbb{N}, \ 1\le i\le 5$.
\end{theo}

\begin{proof}
Let $\c\in S(G_3)$ be written as
\begin{equation}\label{eq:c-represent}
\c=k_1\c_1+k_2\c_2+k_3\c_3+k_4\c_4+k_5\c_5, \qquad k_1,\dots, k_5\in \R.
\end{equation}
Then $\c\in P_h(S(G_3))$ for $h\in \{a,b,c_1,c_2,c_3\}$.

We can use Table \ref{tab:gammarelation} to rewrite $\c$. The results are given in the following table.

\begin{equation*}\label{tab:gammasubs}
\begin{tabular}{l|llrr}
\toprule
Types        &     \ \ \ \ \ \ \ \ \ \ \ \ \ \ \ \ \ \ $\c$           \\
\midrule
$P_a(S(G_3))$         &$k_1\c_1+k_2\c_2+k_3\c_3+k_4\c_4+k_5\c_5$       \\
\midrule
$P_b(S(G_3))$         &$\c_8+(k_1+k_5)\c_1+k_2\c_2+k_3\c_3+(k_4+k_5)\c_4+(-k_5-1)\c_8$  \\
\midrule
$P_{c_1}(S(G_3))$         &$\c_6+(k_1-k_4)\c_1+(k_2+k_4+k_5)\c_2+(k_3+k_4+k_5)\c_3$  \\
              &\ \ \ \ \ \ \ \ \ \ \ \ \  \ \ \ \ \ \ \ \ \ \ \ \ \ \ \  \ \ \ \ \ \ \ \ \ \ \ $+(-k_5-k_4-1)\c_6+(2k_4+k_5)\c_8$  \\
\midrule
$P_{c_2}(S(G_3))$         &$\c_4+\c_6+(k_2+k_1+k_5)\c_2+(k_3+k_1+k_5)\c_3+(k_4-k_1-1)\c_4$     \\
              &\ \ \ \ \ \ \ \ \ \ \ \ \  \ \ \ \ \ \ \ \ \ \ \ \ \ \ \  \ \ \ \ \ \ \ \ \ $+(-k_5-k_1-1)\c_6+(2k_1+k_5)\c_8$  \\
\midrule
$P_{c_3}(S(G_3))$         &$\c_7+(k_2+k_1+k_5)\c_2+(k_3+k_1+k_5)\c_3+(k_4+k_1+k_5)\c_4+k_1\c_6$      \\
              &\ \ \ \ \ \ \ \ \ \ \ \ \  \ \ \ \ \ \ \ \ \ \ \ \ \ \ \  \ \ \ \ \ \ \ \ \ \  $+(-k_5-2k_1-1)\c_7$  \\
\bottomrule
\end{tabular}
\end{equation*}

By Lemma \ref{lem:8-12-2}, $P_h(S(G_3))$ is transformed exactly to type $T_h$ for each $h$.

\end{proof}

\begin{cor}\label{cor-GF-G3}
Let $G_3$ be as above. Then we have the decomposition
$$F^{G_3}(\x,y)=\sum_{\c \in S(G_3)}\x^{\c}y^{s(\c)}=\sum_{i=1}^{5}F_i^{G_3}(\x,y), \qquad\text{where}$$
\begin{enumerate}\label{Solution8-1211}
\item [] $F_1^{G_3}(\x,y)=\frac{1}{(1-\x^{\c_{1}}y)(1-\x^{\c_{2}}y)(1-\x^{\c_{3}}y)(1-\x^{\c_{4}}y^{2})(1-\x^{\c_{5}}y^{2})};$\\
\item [] $F_2^{G_3}(\x,y)=\frac{\x^{\c_{8}}y}{(1-\x^{\c_{8}}y)(1-\x^{\c_{2}}y)(1-\x^{\c_{3}}y)(1-\x^{\c_{1}}y)(1-\x^{\c_{4}}y^{2})};$\\
\item [] $F_3^{G_3}(\x,y)=\frac{\x^{\c_{6}}y}{(1-\x^{\c_{6}}y)(1-\x^{\c_{8}}y)(1-\x^{\c_{2}}y)(1-\x^{\c_{3}}y)(1-\x^{\c_{1}}y)};$\\
\item [] $F_4^{G_3}(\x,y)=\frac{\x^{(\c_{4}+\c_{6})}y^3}{(1-\x^{\c_{3}}y)(1-\x^{\c_{2}}y)(1-\x^{\c_{8}}y)(1-\x^{\c_{6}}y)(1-\x^{\c_{4}}y^{2})};$
\item [] $F_5^{G_3}(\x,y)=\frac{\x^{\c_{7}}y^{2}}{(1-\x^{\c_{6}}y)(1-\x^{\c_{2}}y)(1-\x^{\c_{3}}y)(1-\x^{\c_{4}}y^{2})(1-\x^{\c_{7}}y^{2})}.$
\end{enumerate}

Consequently,
\begin{align*}
 F^{G_3}(y) &=\frac{1+2\,y+4\,{y}^{2}+2\,{y}^{3}+{y}^{4}} {\left({1-y}\right)^{3}\left({1-{y}^{2}}\right)^{2}}=1+5\,y+18\,{y}^{2}+46\,{y}^{3}+101\,{y}^{4}+193\,{y}^{5}+\cdots.
\end{align*}
\end{cor}

\def\diag{\mathop{\mathrm{diag}}}
\def\Oeq{{\mathop{\mathrm{\Omega}}\limits_=}}
\def\Oge{{\mathop{\mathrm{\Omega}}\limits_\geq}}
\section{MacMahon's Partition Analysis and Magic Distinct Labellings
\label{sec:BasicIdeas}}

We first introduce the basic idea of MacMahon's partition analysis and discuss possible
applications of our results.

\subsection{MacMahon's Partition Analysis}

MacMahon's partition analysis was introduced by MacMahon in \cite{MacMahon}, and has been restudied by Andrews and his coauthors in a series of papers
starting with \cite{Andrews}. The Mathematica package \texttt{Omega} was developed in \cite{Andrews-Omega}. The main idea of MacMahon's partition analysis is to replace linear constraints by using new variables and MacMahon's Omega operators on formal series:
\begin{align*}
  \Oeq \sum_{k=-\infty}^\infty c_i \lambda^i &= c_0,\qquad
  \Oge \sum_{k=-\infty}^\infty c_i \lambda^i =\sum_{k=0}^\infty c_i.
\end{align*}
MacMahon's operators always acting on the $\lambda$ variables, which
will be clear from the context.
We explain to how to compute $F^{G_1}(\x,y)$ in Example 1.
By \eqref{eqRegularGraph4-6}, we have
\begin{align*}
 F^{G_1}(\x,y)&= \sum_{(\a,s) \in \N^7}  x_1^{\alpha_1}\cdots x_6^{\alpha_6} y^{s} \Oeq\  \lambda_1^{\alpha_1+\alpha_4+\alpha_5-s} \lambda_2^{\alpha_1+\alpha_2+\alpha_6-s} \lambda_3^{\alpha_2+\alpha_3+\alpha_5-s} \lambda_4^{\alpha_3+\alpha_4+\alpha_6-s} \\
 &=\Oeq\  \frac{\lambda_{{1}} \lambda_{{2}} \lambda_{{3}} \lambda_{{4}} } {D_1(\x,y)}, \text{ where }
\end{align*}
\begin{scriptsize}
\begin{equation*}
D_1(\x,y)=\left({1-\lambda_{{1}}\lambda_{{4}}x_{{1}}}\right)\left({1-\lambda_{{1}}\lambda_{{2}}x_{{2}}}\right)
          \left({1-\lambda_{{2}}\lambda_{{3}}x_{{3}}}\right)\left({1-\lambda_{{3}}\lambda_{{4}}x_{{4}}}\right)
          \left({1-\lambda_{{2}}\lambda_{{4}}x_{{5}}}\right)\left({1-\lambda_{{1}}\lambda_{{3}}x_{{6}}}\right)
          \left({-y+\lambda_{{1}}\lambda_{{2}}\lambda_{{3}}\lambda_{{4}}}\right).
\end{equation*}
\end{scriptsize}
Eliminating the $\lambda$'s will give a representation of $F^{G_1}(\x,y)$. The whole theory relies on unique series expansion of rational functions.
See \cite{xin-fast} for the field of iterated Laurent series and the partial fraction algorithm implemented by the Maple package \texttt{Ell}.
The maple package
\texttt{CTEuclid} in \cite{xin-Euclid} is better in most situations.

The normal form of $F^{G_1}(\x,y)$ (by Maple) already has combinatorial meaning.
The normal form of $F^{G_2}(\x,y)$ is
$$F^{G_2}=\frac{1-{y}^{3}x_{{1}}x_{{2}}x_{{3}}x_{{4}}x_{{5}}x_{{6}}x_{{7}}x_{{8}}x_{{9}}} {\left({1-yx_{{7}}x_{{8}}x_{{9}}}\right)\left({1-yx_{{2}}x_{{6}}x_{{9}}}\right)
\left({1-yx_{{1}}x_{{5}}x_{{8}}}\right)\left({1-yx_{{3}}x_{{4}}x_{{7}}}\right)\left({1-{y}^{2}x_{{1}}x_{{2}}x_{{3}}x_{{4}}x_{{5}}x_{{6}}}\right)},$$
which can be easily decomposed by inspection.

For the graph $G_3$, \texttt{CTEuclid} gives an expression of $F^{G_3}(\x,y)$ quickly, but the normal form of $F^{G_3}(\x,y)$ is
\begin{footnotesize}
\begin{equation*}
F^{G_3}(\x,y)= \frac{N(\x,y)}{(1-yx^{\c_1})(1-yx^{\c_2})(1-yx^{\c_3})(1-y^2x^{\c_4})(1-y^2x^{\c_5})(1-yx^{\c_6})(1-y^2x^{\c_7})(1-yx^{\c_8})},
\end{equation*}
\end{footnotesize}
where
\begin{footnotesize}
\begin{align*}
N(\x,y)=&1-{y}^{3}x_{{1}}x_{{2}}x_{{4}}x_{{5}}x_{{6}}x_{{7}}{x_{{8}}}^{2}{x_{{9}}}^{2}x_{{10}}x_{{12}}
-{y}^{3}x_{{1}}x_{{2}}x_{{3}}x_{{5}}x_{{6}}x_{{7}}x_{{8}}{x_{{9}}}^{2}{x_{{10}}}^{2}x_{{11}}\\
&-2\,{y}^{3}x_{{1}}x_{{2}}x_{{3}}x_{{4}}x_{{5}}x_{{6}}x_{{7}}x_{{8}}x_{{9}}x_{{10}}x_{{11}}x_{{12}}
-{y}^{4}{x_{{2}}}^{2}x_{{3}}x_{{4}}{x_{{5}}}^{2}{x_{{6}}}^{2}{x_{{7}}}^{2}x_{{8}}{x_{{9}}}^{2}x_{{10}}x_{{11}}x_{{12}}\\
&+{y}^{4}{x_{{1}}}^{2}x_{{2}}x_{{3}}x_{{4}}x_{{5}}x_{{6}}x_{{7}}{x_{{8}}}^{2}{x_{{9}}}^{2}{x_{{10}}}^{2}x_{{11}}x_{{12}}
+2\,{y}^{5}x_{{1}}{x_{{2}}}^{2}x_{{3}}x_{{4}}{x_{{5}}}^{2}{x_{{6}}}^{2}{x_{{7}}}^{2}{x_{{8}}}^{2}{x_{{9}}}^{3}{x_{{10}}}^{2}x_{{11}}x_{{12}}\\
&+{y}^{5}x_{{1}}{x_{{2}}}^{2}x_{{3}}{x_{{4}}}^{2}{x_{{5}}}^{2}{x_{{6}}}^{2}{x_{{7}}}^{2}{x_{{8}}}^{2}{x_{{9}}}^{2}x_{{10}}x_{{11}}{x_{{12}}}^{2}
+{y}^{5}x_{{1}}{x_{{2}}}^{2}{x_{{3}}}^{2}x_{{4}}{x_{{5}}}^{2}{x_{{6}}}^{2}{x_{{7}}}^{2}x_{{8}}{x_{{9}}}^{2}{x_{{10}}}^{2}{x_{{11}}}^{2}x_{{12}}\\
&-{y}^{8}{x_{{1}}}^{2}{x_{{2}}}^{3}{x_{{3}}}^{2}{x_{{4}}}^{2}{x_{{5}}}^{3}{x_{{6}}}^{3}{x_{{7}}}^{3}{x_{{8}}}^{3}{x_{{9}}}^{4}{x_{{10}}}^{3}{x_{{11}}}^{2}{x_{{12}}}^{2}, \end{align*}
\end{footnotesize}
is polynomial of 10 terms. It is not clear how to decompose $F^{G_3}(\x,y)$ as a sum of simple rational functions.
We guessed such a decomposition (in Corollary \ref{cor-GF-G3}) by certain criterion. The verification of the formula
by computer is easy.

We should mention that for some complicated graphs $G$, Maple will stuck when normal $F^G(\x,y)$.

We conclude the subsection by reporting the following result.
Let $G_4$ be given in Figure \ref{fig:RegularGraph6-9-2}, with 6 vertices and 9 edges.

\begin{figure}[h]
  \centering
  \includegraphics[width=3.6cm]{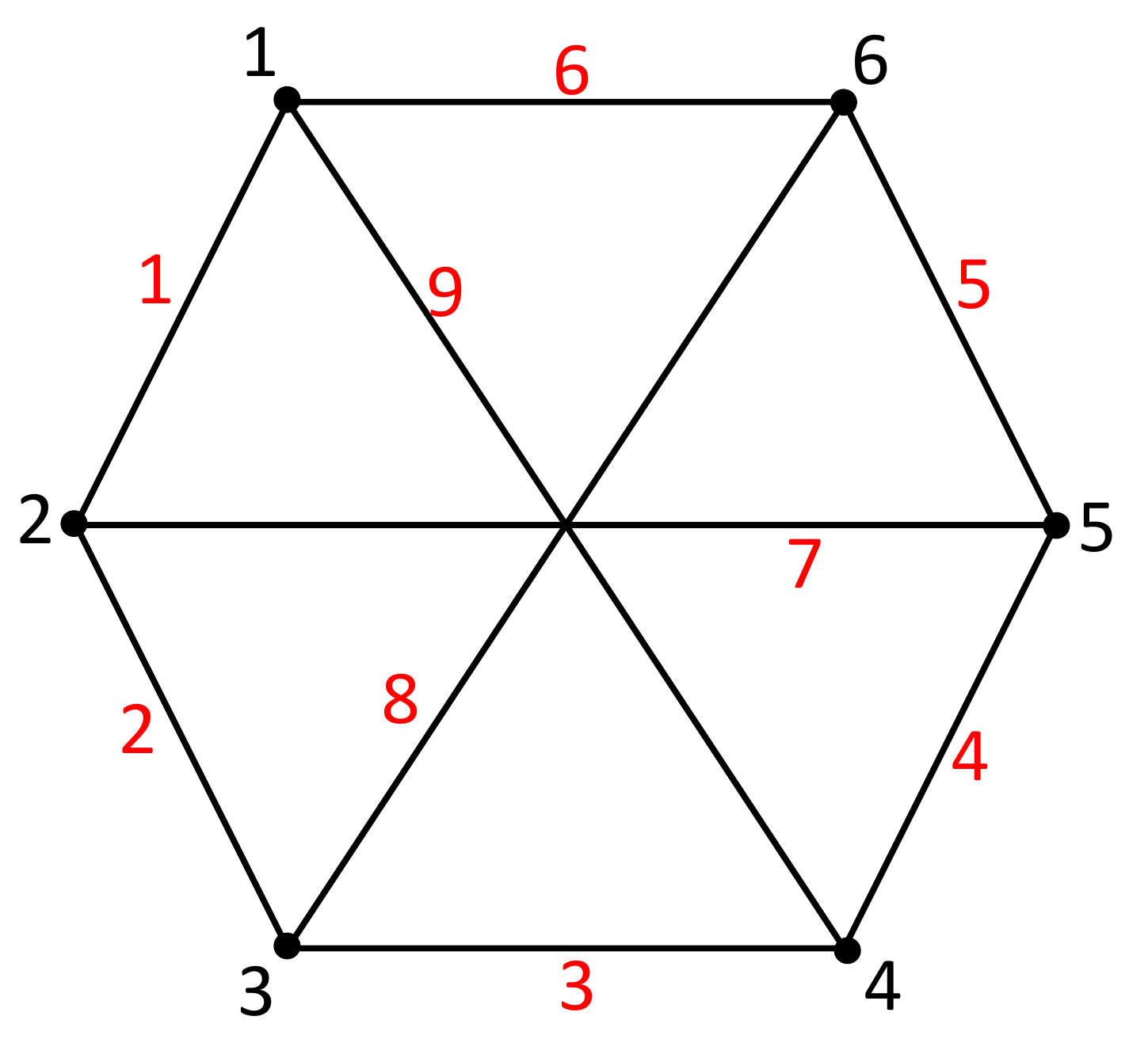}\vspace{-0.5cm}\\
  \caption{Regular graph $G_4$ with 6 vertices and 9 edges.
}\label{fig:RegularGraph6-9-2}
\end{figure}

Then
\begin{equation}\label{G4Decomposition}
\begin{split}
  F^{G_4}(\x,y)&= \frac{y x_{{2}} x_{{4}} x_{{6}} } {\left({1-yx_{{1}}x_{{4}}x_{{8}}}\right)\left({1-yx_{{2}}x_{{4}}x_{{6}}}\right)\left({1-yx_{{2}}x_{{5}}x_{{9}}}\right)\left({1-yx_{{3}}x_{{6}}x_{{7}}}\right)\left({1-yx_{{1}}x_{{3}}x_{{5}}}\right)}\\
  & +\frac{1} {\left({1-yx_{{1}}x_{{4}}x_{{8}}}\right)\left({1-yx_{{2}}x_{{5}}x_{{9}}}\right)\left({1-yx_{{3}}x_{{6}}x_{{7}}}\right)\left({1-yx_{{7}}x_{{8}}x_{{9}}}\right)\left({1-yx_{{1}}x_{{3}}x_{{5}}}\right)}\\
  &+\frac{y ^{2}x_{{2}} x_{{4}} x_{{6}} x_{{7}} x_{{8}} x_{{9}} } {\left({1-yx_{{7}}x_{{8}}x_{{9}}}\right)\left({1-yx_{{3}}x_{{6}}x_{{7}}}\right)\left({1-yx_{{2}}x_{{5}}x_{{9}}}\right)\left({1-yx_{{2}}x_{{4}}x_{{6}}}\right)\left({1-yx_{{1}}x_{{4}}x_{{8}}}\right)}.
\end{split}
\end{equation}
\begin{align}\label{G4-Fy}
 F^{G_4}(y) &=\frac{1+y+{y}^{2}} {\left({1-y}\right)^{5}}=1+6\,y+21\,{y}^{2}+55\,{y}^{3}+120\,{y}^{4}+231\,{y}^{5}+406\,{y}^{6}+\cdots.
\end{align}

It is not hard to give a combinatorial proof using similar ideas.

\subsection{Magic Labellings}
The complete generating function $F^G(\x,y)$ encodes almost all information of $S(G)$.

\medskip
Let $S^*(G)$ be the set of magic distinct labellings of $G$.
In some literature,
the technique condition ``positive" is added because it
is possible that any $\a \in S(G)$ must have label $0$ on some edges. For instance,
if $G$ is given by Figure \ref{fig:RegularGraph5-6 and fig:RegularGraph4-4} (a), then it is easy to check that $S(G)$ only contains the all $0$ labelling;
if $G$ is given by Figure \ref{fig:RegularGraph5-6 and fig:RegularGraph4-4} (b), then for any $\a \in S(G)$ the labels of 2 and 3 must be $0$. In deed,
we have $
F^{G_{5(b)}}(\x,y)= \frac{1} {1-yx_{{1}}x_{{4}}}.
$

\begin{figure}[htbp]
\centering
\subfigure[]{
\begin{minipage}[t]{0.25\linewidth}
\centering
\includegraphics[width=1.5in]{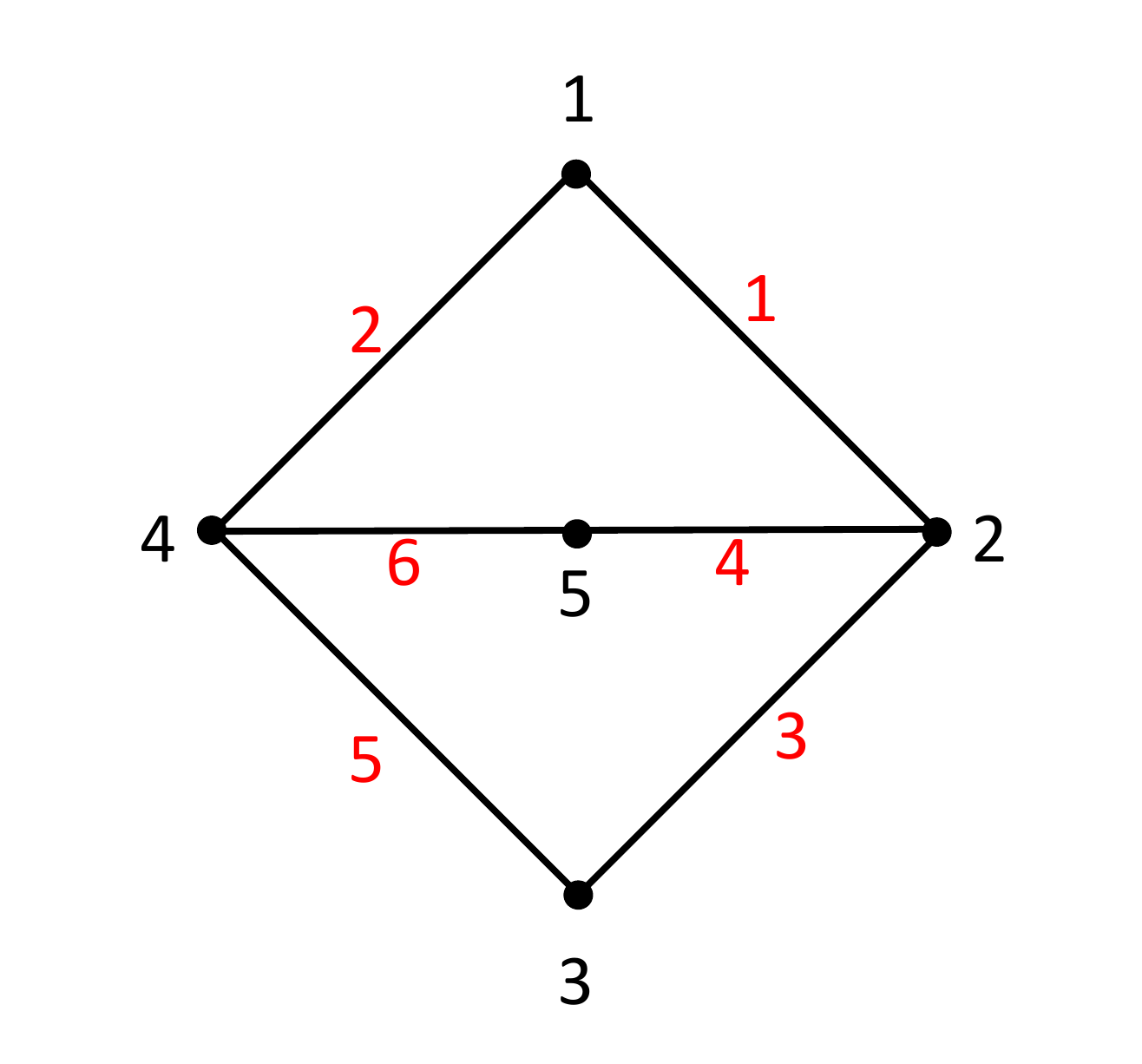}
\end{minipage}%
}\ \ \ \
\subfigure[]{
\begin{minipage}[t]{0.25\linewidth}
\centering
\includegraphics[width=1.5in]{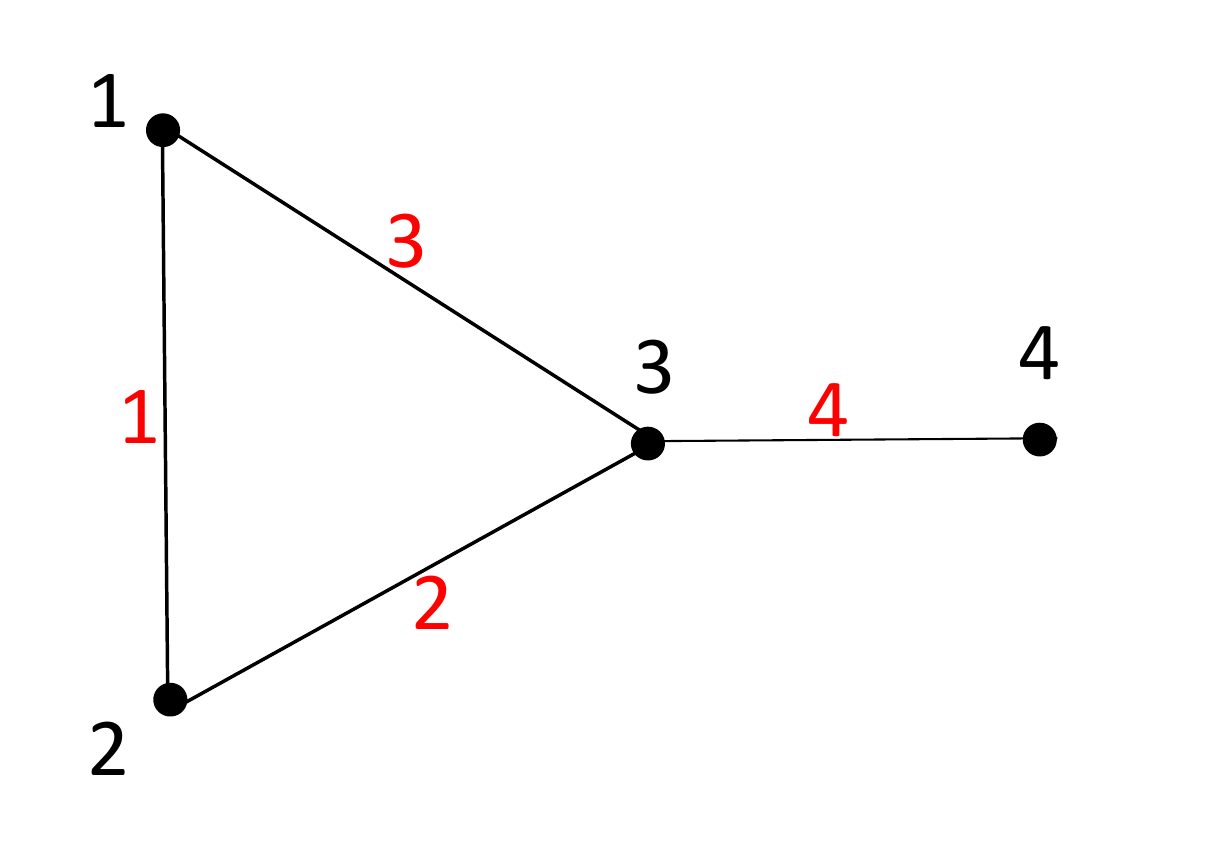}
\end{minipage}%
}%

\centering
\caption{Graph $G_{5(a)}$ with 5 vertices and 6 edges and graph $G_{5(b)}$ with 4 vertices and 4 edges.
}\label{fig:RegularGraph5-6 and fig:RegularGraph4-4}
\end{figure}

For a $d$ regular graph $G$, the all $1$
labelling $\mathbf{1}$ is magic with magic sum $d$. Thus $S_\P(G)= \mathbf{1}+S(G)=\{\mathbf{1}+\a: \a \in S(G)\}$.
In this sense, the positive condition makes no difference for magic labellings of
a regular graph.

The graphs $G_1,G_2,G_3$ in our examples are all regular graphs.
It is an accident that none of them have magic distinct labellings.
Indeed, they do not have magic distinct $\R$-labellings. To see this for $G_3$,
any $\c\in S_\R(G_3)$ can be written as in \eqref{e-gamma} for some $k_1,\dots, k_5\in \R$. Then the $a_2$ and $a_6$ labels
are the same. The situation for the other two graphs are similar: look at the
$a_1$ and $a_3$ labels for $G_1$, and the $a_1$ and $a_5$ labels for $G_2$.

In general, the structure of $S^*(G)$ is pretty complicated. It is obtained by slicing out
all $\a\in S(G)$ in the $\binom{n}{2}$ hyper planes $\alpha_i=\alpha_j, \ 1\le i<j \le n$.
Using inclusion and exclusion principle
will be too expensive since that will involve $2^{\binom{n}{2}}$ cases.
It is possible to obtain the generating function
$$E^G(\x,y)=\sum_{\a\in S^*(G)} \x^\a y^{s(G)} $$
by MacMahon's partition analysis.

Here we introduce two operators that can be realized by MacMahon's partition analysis.
If $A(x,y)=\sum_{i,j\ge 0}a_{i,j} x^i y^j$ is a formal power series in $x$ and $y$. Then the diagonal operator
defined by
$$ \diag_{x,y} A(x,y)= \sum_{i\ge 0}a_{i,i} x^i y^i =\sum_{i,j\ge 0, \ i-j=0} a_{i,j} x^i y^j   $$
can be realized by MacMahon's Omega (linear) operator.
We have
$$\diag_{x,y} A(x,y) =\sum_{i,j\ge 0} a_{i,j}   x^i y^j \Oeq\; \lambda^{i-j} =\Oeq \sum_{i,j\ge 0} a_{i,j}   (\lambda x)^i (y/\lambda)^j
=\Oeq A(\lambda x, y/\lambda).
$$

Similarly, if we define
$$ \diag_{x>y} A(x,y)= \sum_{i>j \ge 0}a_{i,j} x^i y^j =\sum_{i,j\ge 0, \ i-j-1\ge 0} a_{i,j} x^i y^j,   $$
then it can be realized by
$$\diag_{x>y} A(x,y) =\sum_{i,j\ge 0} a_{i,j}   x^i y^j \Oge \lambda^{i-j-1} =\Oge \sum_{i,j\ge 0} a_{i,j} \lambda^{-1}  (\lambda x)^i (y/\lambda)^j
=\Oge \lambda^{-1} A(\lambda x, y/\lambda).
$$

We have
$$E^G(\x,y)= \prod_{1\le i<j\le n} (1-\diag_{x_i,x_j})  F^G(\x,y),$$
whose expansion is just the inclusion and exclusion result.
We can normal the result after each application of $1-\diag_{x_i,x_j}$, provided that the
result would not explode, i.e., the numerator becomes too large for Maple to handel. Note that the computation highly relies on
the order of the operators. For instance, $(1-\diag_{x_2,x_6}) F^{G_3}(\x,y)=0$.
But the result quickly explodes for some orders.

Another way is to use the natural
decomposition
$$ S^*(G)= \biguplus_{\pi \in \mathfrak{S}_n} S^\pi(G),$$
where $\pi$ ranges over all permutations of $\{1,2,\dots,n\}$, and $S^\pi(G)$
consists of all $\a=(\alpha_1,\dots, \alpha_n)\in S(G)$ compatible with $\pi$, i.e.,
$\alpha_{\pi_1}>\alpha_{\pi_2}>\cdots >\alpha_{\pi_n}$.

The generating function $E^{G,\pi}(\x,y)$ of $S^\pi(G)$ can be extracted from $F^G(\x,y)$ by
applying MacMahon's Omega operator. We have
$$ E^{G,\pi}(\x,y) = \diag_{x_{\pi_1}>x_{\pi_2}} \cdots \diag_{x_{\pi_{n-1}}>x_{\pi_n}} F^G(\x,y). $$

For $F^{G_4}(\x,y)$ with the combinatorial decomposition as $F_1+F_2+F_3$ as in \eqref{G4Decomposition},
and for each $F_i, \ i=1,2,3$, we can extract
$$ E_i^\pi =  \diag_{x_{\pi_1}>x_{\pi_2}} \cdots \diag_{x_{\pi_{8}}>x_{\pi_9}} F_i(\x,y).$$
for any particular $\pi\in \mathfrak{S}_9$. Only $432=2^43^3$ out of $362880=9!$ permutations
give non-varnishing results. And the three sets of permutations do not overlap.
 Each result are simple rational functions with numerator either a monomial
or a binomial. For instance,

\begin{equation*}
E_1^{134568279} = \frac{x_{{4}} ^{7}x_{{2}} ^{3}x_{{6}} ^{5}y ^{15}x_{{1}} ^{10}x_{{3}} ^{8}x_{{5}} ^{6}x_{{8}} ^{4}x_{{7}} ^{2}\left({1-{x_{{4}}}^{7}{x_{{2}}}^{3}{x_{{6}}}^{5}{y}^{15}{x_{{1}}}^{10}{x_{{3}}}^{8}{x_{{5}}}^{6}{x_{{8}}}^{4}{x_{{7}}}^{2}}\right)} {D(\x,y)},
\end{equation*}
where
\begin{align*}
D(\x,y)=&\left({1-{y}^{3}x_{{1}}x_{{2}}x_{{3}}x_{{4}}x_{{5}}x_{{6}}x_{{7}}x_{{8}}x_{{9}}}\right)
        \left({1-{y}^{4}{x_{{1}}}^{3}x_{{2}}{x_{{3}}}^{2}{x_{{4}}}^{2}{x_{{5}}}^{2}x_{{6}}x_{{8}}}\right)\\
        &\times\left({1-{y}^{5}{x_{{1}}}^{3}x_{{2}}{x_{{3}}}^{3}{x_{{4}}}^{2}{x_{{5}}}^{2}{x_{{6}}}^{2}x_{{7}}x_{{8}}}\right)
        \left({1-{y}^{6}{x_{{1}}}^{4}x_{{2}}{x_{{3}}}^{3}{x_{{4}}}^{3}{x_{{5}}}^{2}{x_{{6}}}^{2}x_{{7}}{x_{{8}}}^{2}}\right)\\
        &\times\left({1-{y}^{7}{x_{{1}}}^{5}x_{{2}}{x_{{3}}}^{4}{x_{{4}}}^{3}{x_{{5}}}^{3}{x_{{6}}}^{2}x_{{7}}{x_{{8}}}^{2}}\right)
        \left({1-{y}^{8}{x_{{1}}}^{5}{x_{{2}}}^{2}{x_{{3}}}^{4}{x_{{4}}}^{4}{x_{{5}}}^{3}{x_{{6}}}^{3}x_{{7}}{x_{{8}}}^{2}}\right).
\end{align*}

There are total of $2^43^4=1296$ permutations $\pi$ such that $E^{G_4,\pi}(\x,y)\neq 0$.

As a consequence,
we have
\begin{align}\label{G4Egenerationfuction}
  &E^{G_4}(y)=\frac{72  y ^{12}\left({1-y}\right)^{2} N_4(y)} {\left({1-{y}^{3}}\right)^{2}\left({1-{y}^{4}}\right)\left({1-{y}^{5}}\right)\left({1-{y}^{6}}\right)\left({1-{y}^{7}}\right)\left({1-{y}^{8}}\right)}\\
 &= 72({y}^{12}+2\,{y}^{13}+4\,{y}^{14}+8\,{y}^{15}+12\,{y}^{16}+20\,{y}^{17}
+29\,{y}^{18}+42\,{y}^{19}+54\,{y}^{20}
+\cdots) \nonumber
\end{align}
where $N_4(y)$ is given by
\begin{align*}
&1+4\,y+11\,{y}^{2}+24\,{y}^{3}+44\,{y}^{4}+73\,{y}^{5}+109\,{y}^{6}+152\,{y}^{7}+192\,{y}^{8}+233\,{y}^{9}+258\,{y}^{10}\\
      &+274\,{y}^{11}+268\,{y}^{12}+249\,{y}^{13}+207\,{y}^{14}+166\,{y}^{15}+117\,{y}^{16}+79\,{y}^{17}+41\,{y}^{18}+18\,{y}^{19}.
\end{align*}

The 72 in the numerator seems a surprise. Observe that the symmetry group of $G_4$ is the Dihedral group $D_6$ which
is of cardinality $12$. Thus we have the following result.
\begin{cor}\label{cor-G4}
  Let $\widetilde{h}(s)$ be the number of magic distinct labellings of $G_4$ of magic sum $s$ up to isomorphism under the Dihedral group $D_6$. Then
it is divisible by $6$ for all $s$.
\end{cor}
\begin{proof}
Clearly we have
$$E^{G_4}(y)/12= \sum_{s\ge 0} \widetilde{h}(s) y^s. $$
The corollary then follows by the formula \eqref{G4Egenerationfuction}.
\end{proof}

Corollary \ref{cor-G4} needs a combinatorial proof.
We list in Figure  \ref{G4:1-6} all non-isomorphic magic labellings of $G4$ with minimum magic sum $s=12$.

\begin{figure}[htbp]
\centering
\subfigure[]{
\begin{minipage}[t]{0.25\linewidth}
\centering
\includegraphics[width=1in]{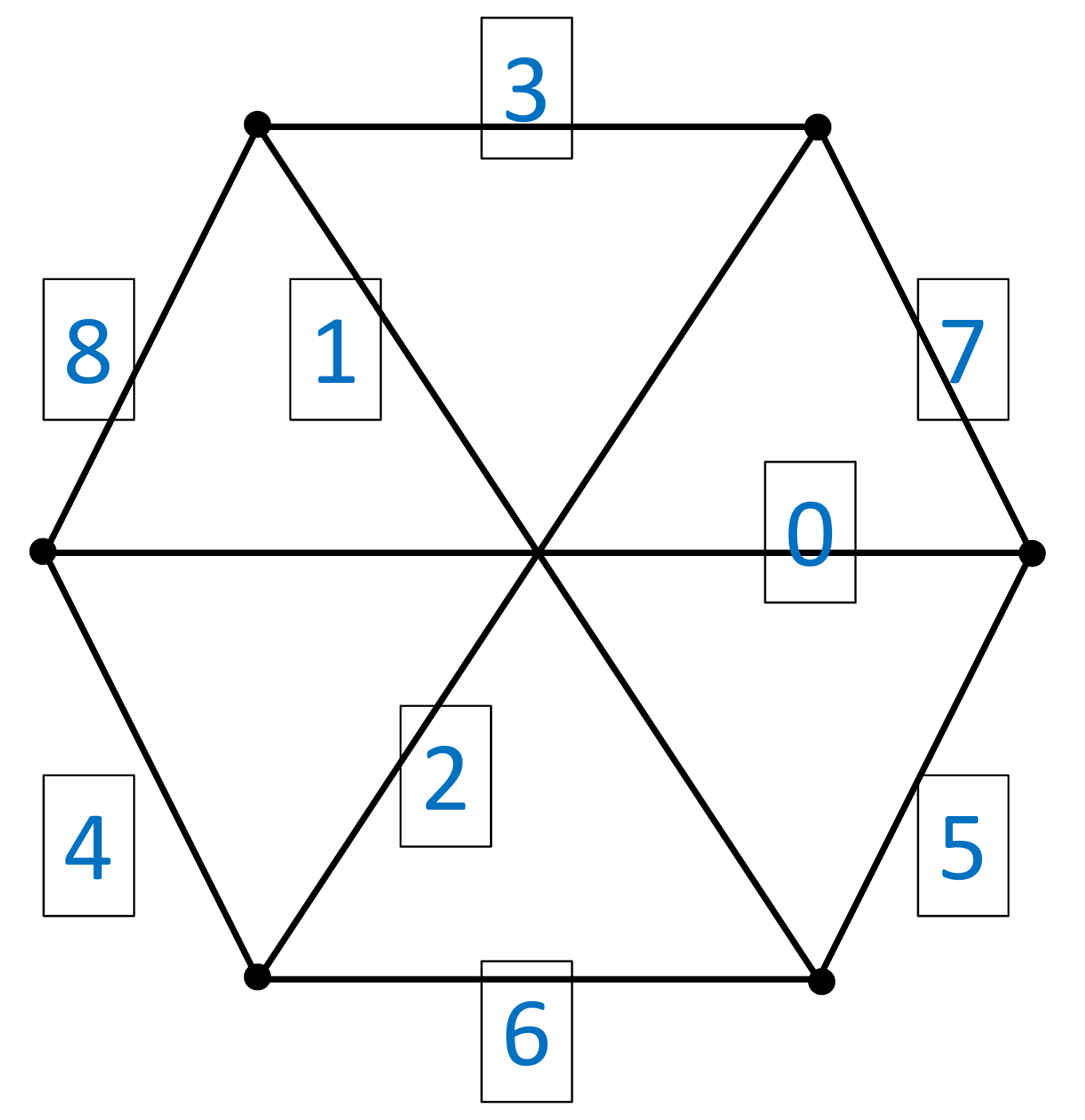}
\end{minipage}%
}
\subfigure[]{
\begin{minipage}[t]{0.25\linewidth}
\centering
\includegraphics[width=1in]{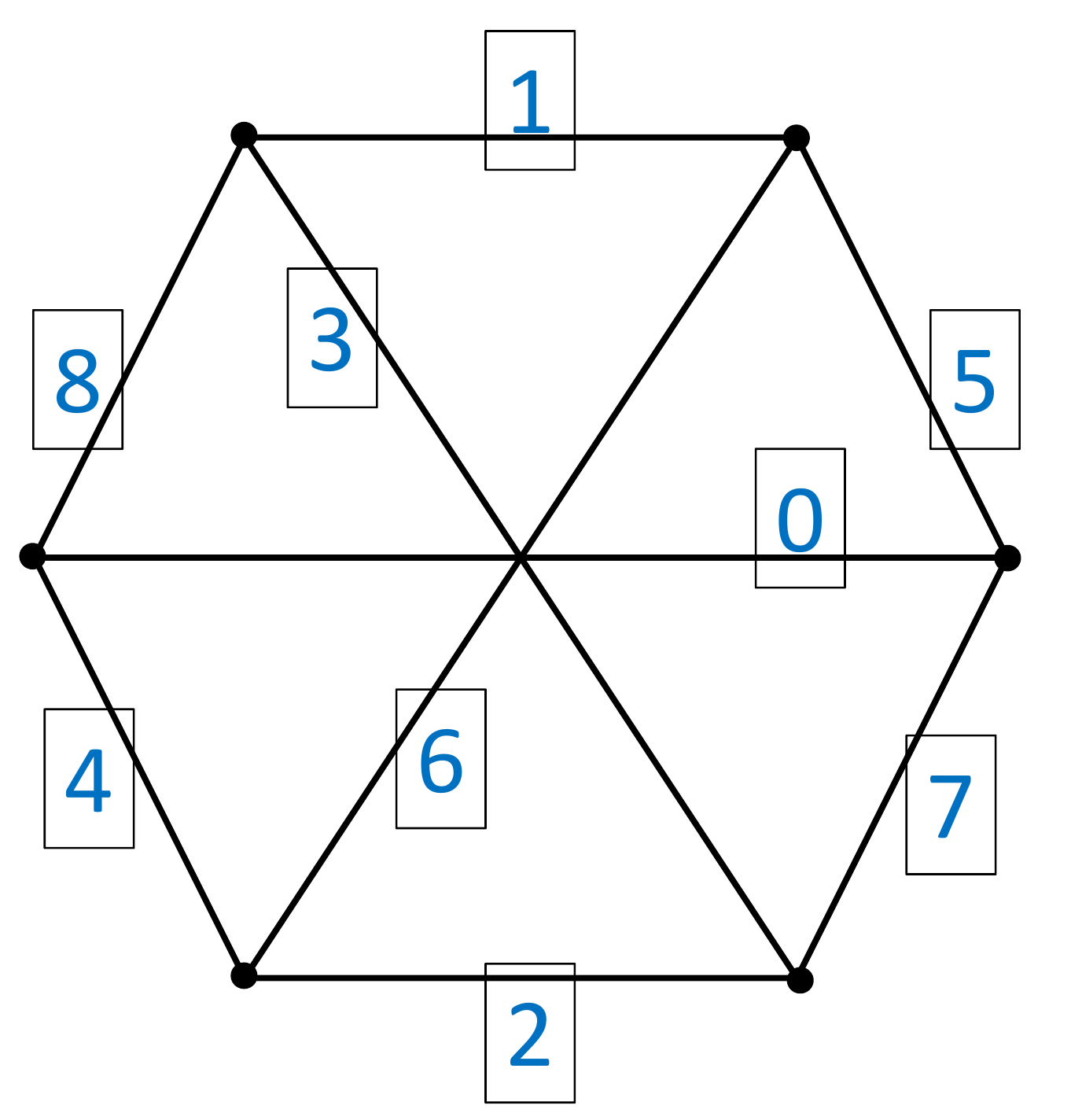}
\end{minipage}%
}%
\subfigure[]{
\begin{minipage}[t]{0.25\linewidth}
\centering
\includegraphics[width=1in]{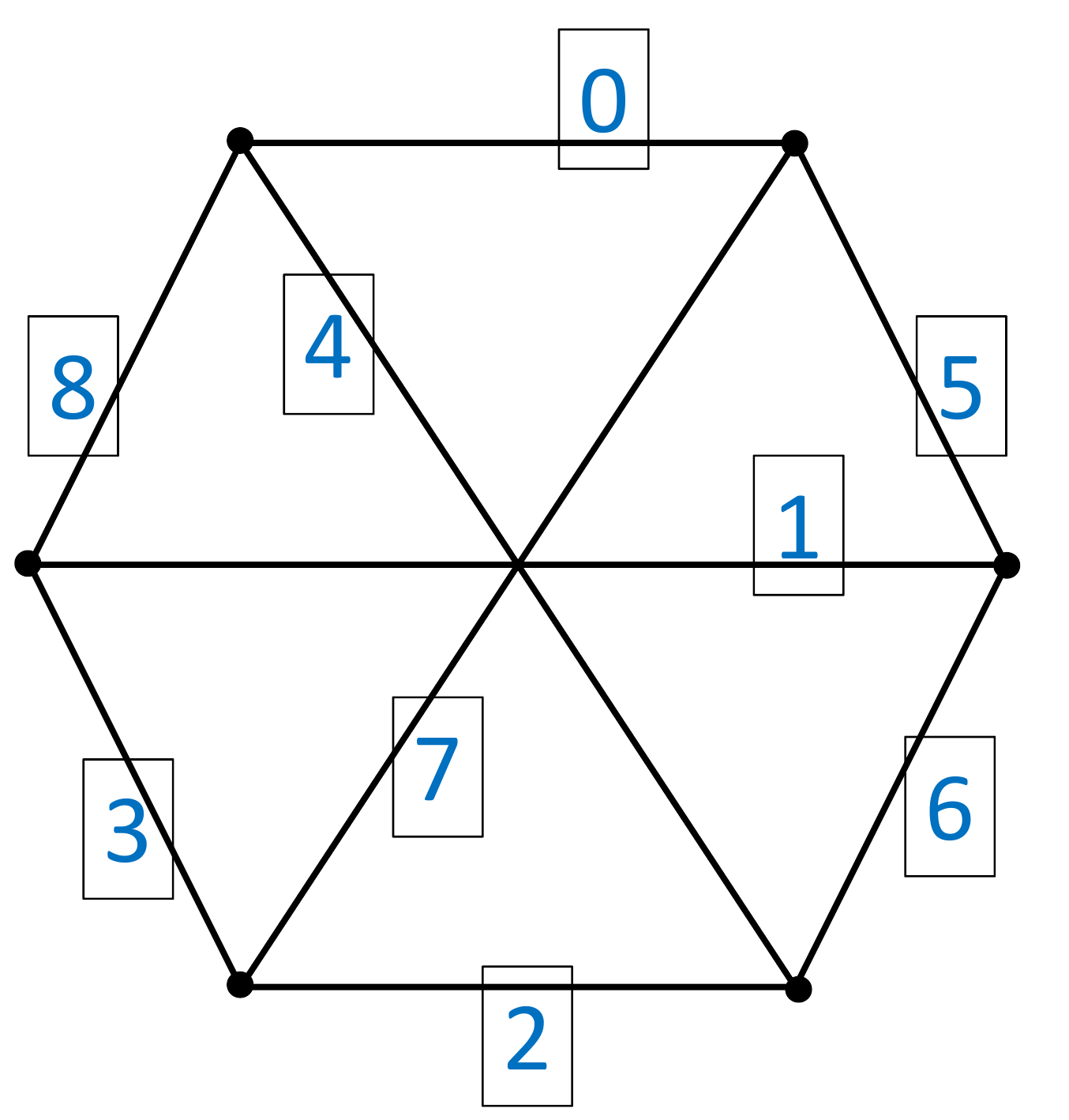}
\end{minipage}
}\\
\subfigure[]{
\begin{minipage}[t]{0.25\linewidth}
\centering
\includegraphics[width=1in]{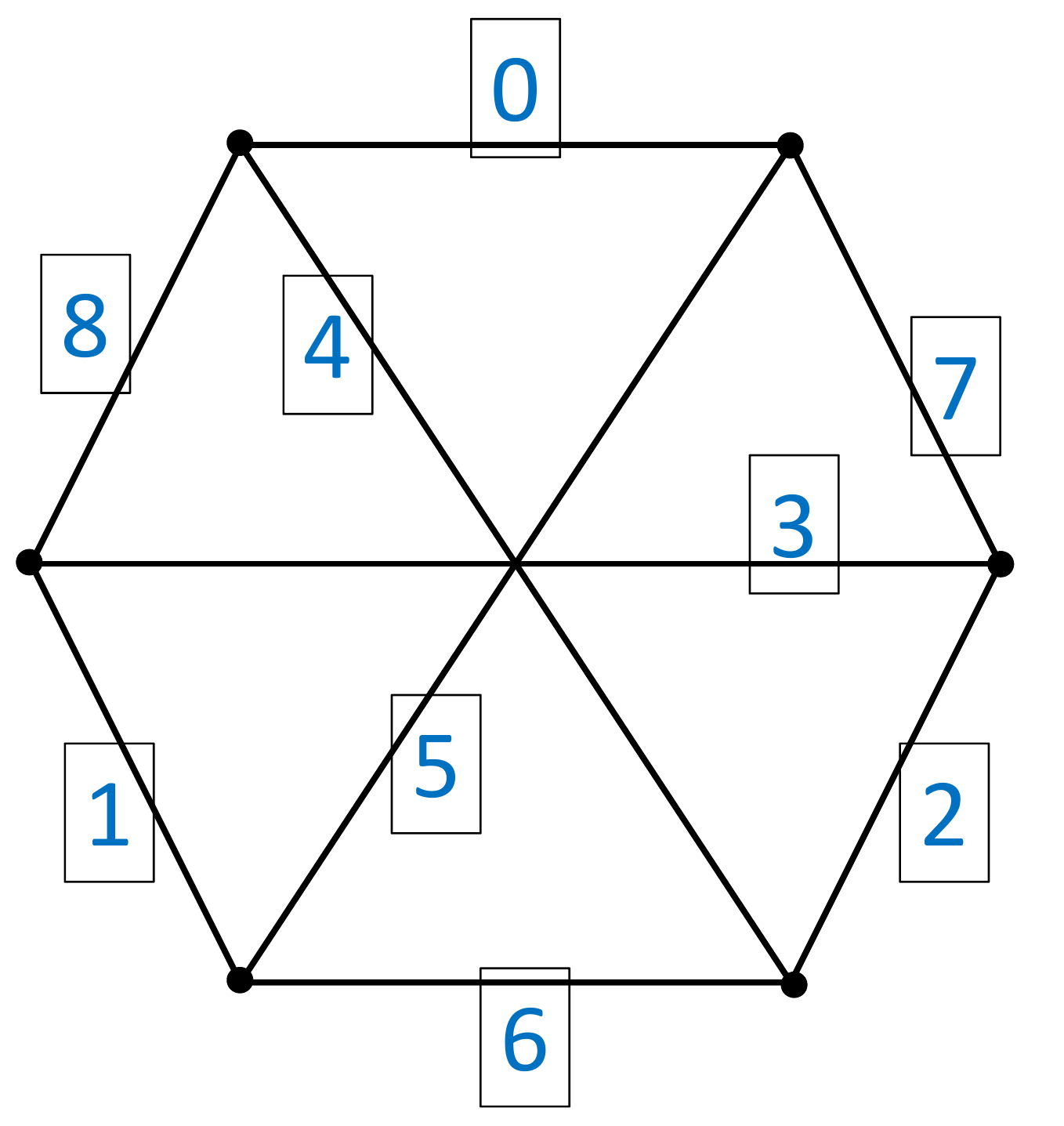}
\end{minipage}%
}%
\subfigure[]{
\begin{minipage}[t]{0.25\linewidth}
\centering
\includegraphics[width=1in]{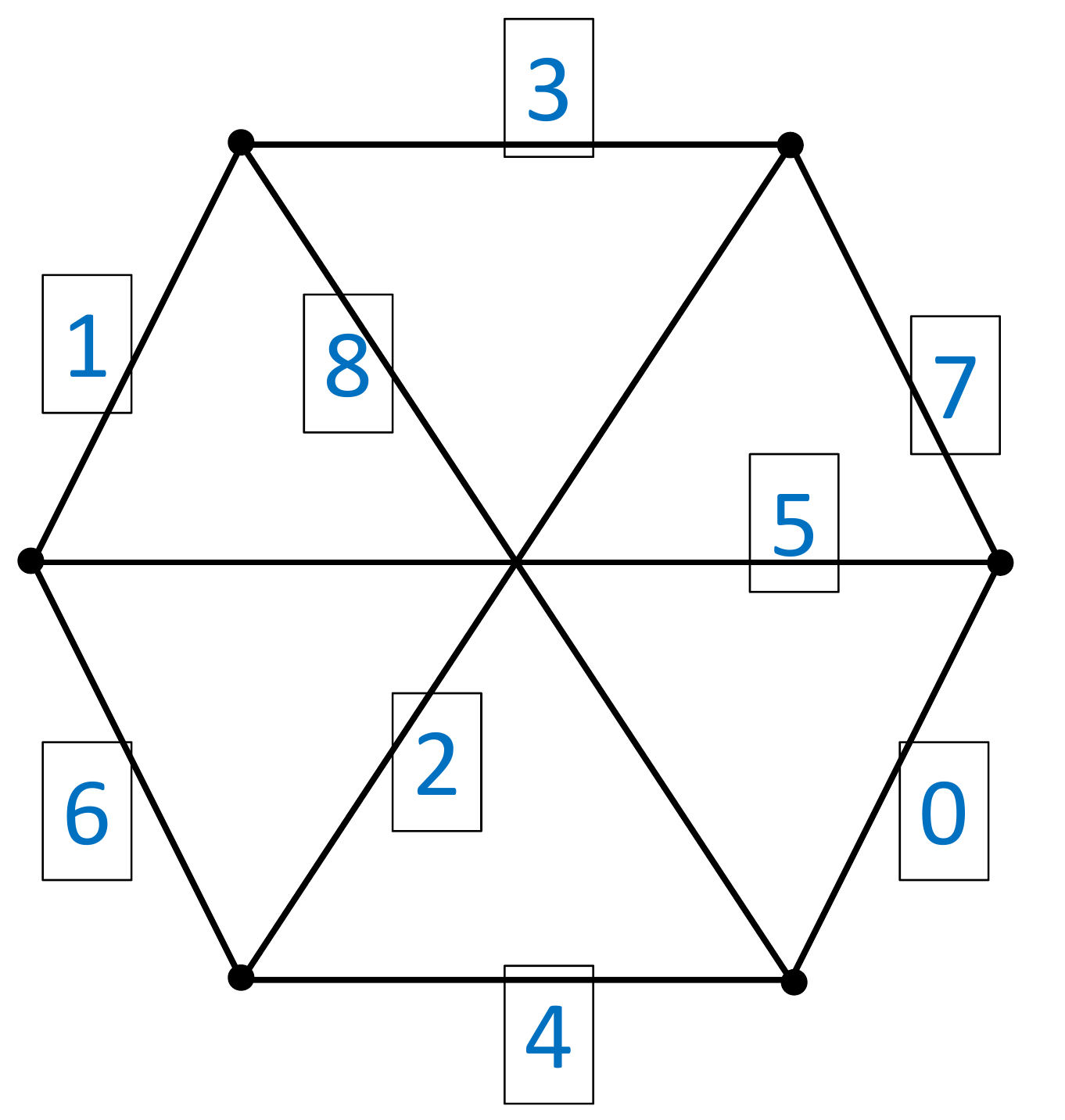}
\end{minipage}%
}%
\subfigure[]{
\begin{minipage}[t]{0.25\linewidth}
\centering
\includegraphics[width=1in]{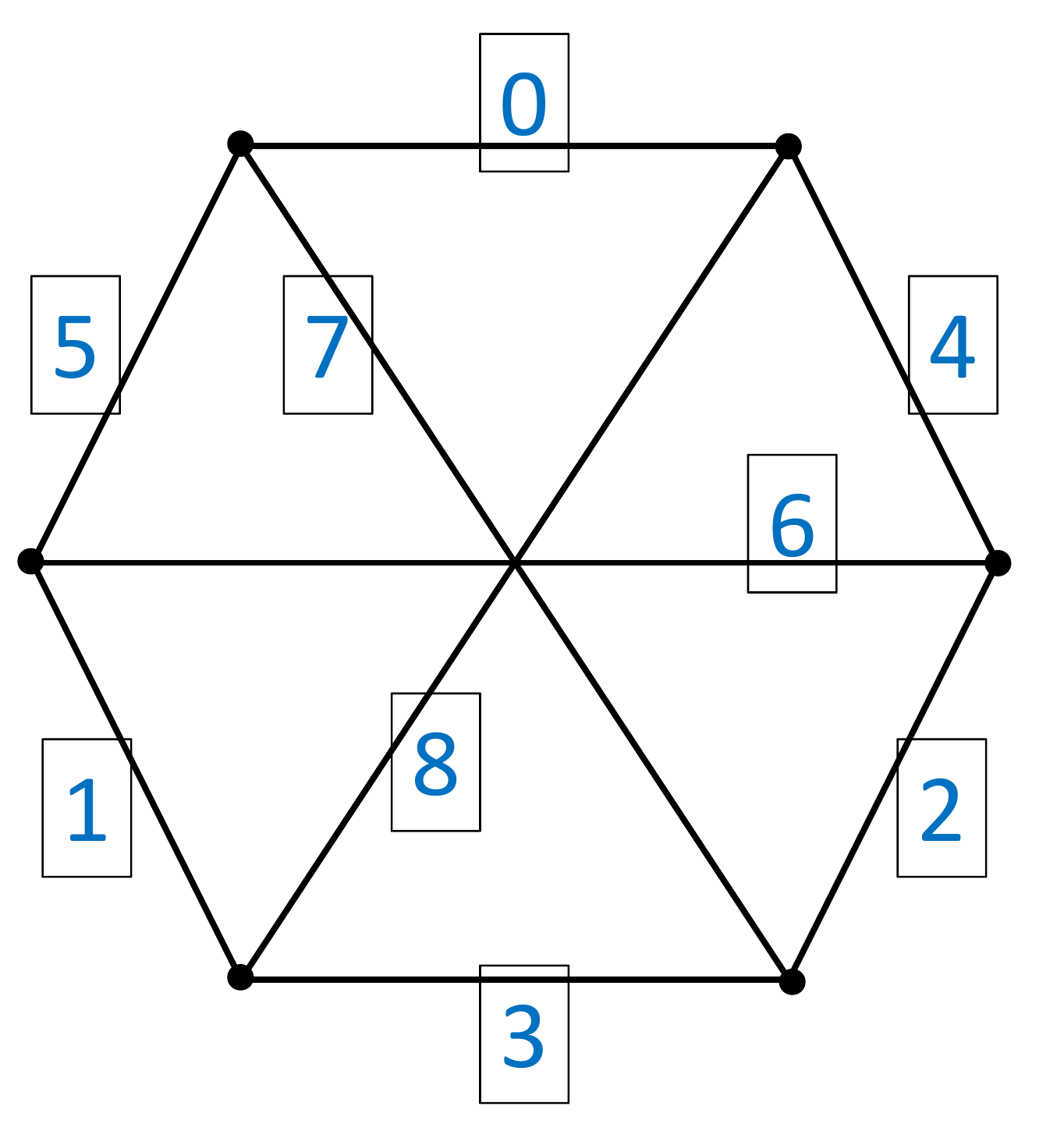}
\end{minipage}
}%
\centering
\caption{All non-isomorphic magic labellings of $G_4$ with minimum magic sum $s=12$.
The edges are labeled by the blue boxed numbers.
}\label{G4:1-6}
\end{figure}

\section{Concluding Remark \label{sec:Remark}}
We have studied the complete construction of magic labelling of graphs $S(G)$. Our aim is to decompose
$S(G)$ into some shifted free monoids. We have achieved this for four graphs, and give
combinatorial proofs of the decompositions.

In general, there are algorithms to compute the generating functions $F^G(\x,y)$. Then the decomposition
corresponds to algebraic decomposition of $F^G(\x,y)$. Such a decomposition seems easier to attack, and it is
a guide for combinatorial proofs.

Our approach to magic distinct labellings is by using MacMahon's partition analysis, especially the Maple
package \texttt{CTEuclid}. The package extracts constant term of an Elliott-rational function, i.e.,
a rational function whose denominator is a product of binomials. The number of binomials in the denominator
 affects the performance
of Maple significantly. This is why we prefer a good decomposition of $F^G(\x,y)$.

Magic distinct labellings of the cube was studied in \cite{Xin-cube}, where the cube has 8 vertices and 12 edges.
The generating function is more complicated than
that of $G_4$. It seems that the more edges the graph have, the more complex the generating function is.

Another direction is to restrict the number of vertices.
In an upcoming paper, we will report the results for all graphs with 5 vertices.


\begin{thebibliography}{2}
\bibitem{Andrews}
G. E. Andrews, MacMahon¡¯s partition analysis. I. The lecture hall partition theorem, Mathematical essays in honor of Gian-Carlo Rota (Cambridge, MA, 1996), Progr. Math., vol.161, Birkh\"{a}user Boston, Boston, MA, 1998, pp. 1--22.

\bibitem{Andrews-Omega} 
G.E. Andrews, P. Paule, A. Riese, MacMahon's partition analysis: the Omega package, European Journal of Combinatorics, 2001, 22(7):887--904.


\bibitem{Baker2}
A. Baker, J. Sawada, Magic Labelings on Cycles and Wheels, International Conference on Combinatorial Optimization \& Applications, Springer-Verlag, 2008.

\bibitem{Chedomir}
C.A. Barone, Magic labelings of directed graphs, available to the world wide web, 2008.

\bibitem{Matthias1}
 M. Beck, T. Zaslavsky, An enumerative Geometry for magic and magilatin labellings, Annals of Combinatorics, 2006, 10(4):395--413.

\bibitem{Arnold}
 A.A. Eniego, I. Garces, Characterization of completely $k$-magic regular graphs, Journal of Physics Conference Series, 2016, 893(1).

\bibitem{Matthias2}
R.N. Jamil, M.A. Rehman, M. Javaid, Partially magic labelings and the Antimagic Graph Conjecture, S\'{e}minaire Lotharingien de Combinatoire, 2017.

\bibitem{Kotzig}
 A. Kotzig, A. Rosa, Magic valuations of finite graphs, Canadian Mathematical Bulletin, 1970, 13:451--461.


\bibitem{MacDougall}
J.A. MacDougall, M. Miller, Slamin, W.D. Wallis, Vertex-magic total labelings graphs, Utilitas Mathematica, 2002, 61, 3--21.

\bibitem{MacMahon}
P. A. MacMahon, Combinatory Analysis, vol. 2, Cambridge University Press, Cambridge,
1915¨C1916, Reprinted: Chelsea, New York, 1960.


\bibitem{Nissankara}
N.L. Prasanna, N. Sudhakar, Algorithms for magic labeling on graphs, Journal of Theoretical \& Applied Information Technology, 2014, 66(1):36--42.

\bibitem{Stanley-E-Combinatiorics1}
R. Stanley, {Enumerative Combinatorics}, Volume 1, Cambridge University Press, Cambridge, 1997.

\bibitem{Stanley-magiclabelings}
 R. Stanley, Linear homogeneous Diophantine equations and magic labelings of graphs, Duke Math. J., 1973, 40,
607--632.


\bibitem{West}
 D.B. West, Introduction to Graph Theory, Prentice-Hall, Upper Saddle River, NJ, 1996.


\bibitem{xin-Euclid} 
G. Xin, A Euclid style algorithm for MacMahon's partition analysis, Journal of Combinatorial Theory Series A, 2015.


\bibitem{xin2}
G. Xin, Constructing all magic squares of order three, Discrete Mathematics, 2008, 308(15):3393--3398.


\bibitem{xin-fast}
G. Xin, A fast algorithm for MacMahon's partition analysis, Electronic Journal of Combinatorics, 2004, 11(1):58.


\bibitem{Xin-cube}
G. Xin, Y.R. Zhang, and Z.H. Zhang, On Magic Distinct Labellings of the Cube, in preparation.










\end{thebibliography}
\end{document}